\documentclass{amsart}

\usepackage{amsmath,amsthm,amssymb}
\usepackage{mathtools}
\usepackage{tikz,tikz-3dplot,tikz-cd,graphicx}
\usepackage{youngtab}
\usepackage{adjustbox}
\usetikzlibrary{shapes.geometric}
\graphicspath{
  {./}
  {figures/}
  {figures/rectangles/}
  {figures/deg5example/}
  {figures/annuli/}
  {figures/polynomials/}
  {figures/brrect/}
}
\usepackage[hidelinks]{hyperref}
\usepackage{standalone}

\theoremstyle{plain}
\newtheorem{thm}{Theorem}[section]

\newtheorem{prop}[thm]{Proposition}

\newtheorem{mainthm}{Theorem}

\theoremstyle{definition}
\newtheorem{defn}[thm]{Definition}
\newtheorem{rem}[thm]{Remark}

\newcommand{\mathsc}[1]{{\normalfont\textsc{#1}}}

\newcommand{\R}{\mathbb{R}}

\newcommand{\Ib}{\mathbf{I}}\newcommand{\Jb}{\mathbf{J}}
\renewcommand{\S}{\mathbb{S}}

\newcommand{\bb}{\mathbf{b}}

\newcommand{\Qb}{\mathbf{Q}}

\newcommand{\mult}{\textsc{Mult}}

\newcommand{\sym}{\textsc{Sym}}

\newcommand{\conf}{\textsc{Conf}}

\newcommand{\poly}{\textsc{Poly}}

\newcommand{\bool}{\textsc{Bool}}

\newcommand{\bt}{\begin{tabular}}
\newcommand{\et}{\end{tabular}}

\newcommand{\comp}{\mathsc{Comp}}

\newcommand{\bv}{\mathbf{v}}
\newcommand{\ba}{\mathbf{a}}
\newcommand{\bm}{\mathbf{m}}
\newcommand{\bx}{\mathbf{x}}
\newcommand{\bz}{\mathbf{z}}

\newcommand{\by}{\mathbf{y}}

\newcommand{\closedint}{
  \begin{tikzpicture}[baseline]
    \draw[thick] (0,.1)--(.4,.1);
    \filldraw[color=black,fill=black!70!white] (0,.1) circle (.3mm) (.4,.1) circle (.3mm);
  \end{tikzpicture}
}

\newcommand{\closedsquare}{\protect
  \begin{tikzpicture}[baseline]
    \filldraw[thick,fill=white!80!black] (0,0) rectangle (.2,.2);
  \end{tikzpicture}
}

\newcounter{joncomments}

\newcounter{michaelcomments}

\begin{document}

\title{The Geometry of Rectangular Multisets}
\author{Michael Dougherty} 
\email{doughemj@lafayette.edu}
\address{Department of Mathematics, Lafayette College,
  Easton, PA 18042}
\author{Jon McCammond}
\email{jon.mccammond@math.ucsb.edu}
\address{Department of Mathematics, UC Santa Barbara, 
  Santa Barbara, CA 93106} 
\date{\today}

\begin{abstract}
    This article describes a natural piecewise Euclidean
    bi-simplicial cell structure for the space of $n$-element multisets
    in a fixed Euclidean rectangle. In particular, we highlight 
    some connections with spaces of complex polynomials and permutahedra. 
\end{abstract}

\maketitle

\section*{Introduction}

For each topological space $X$ and positive integer $n$, the symmetric 
group $\sym_n$ acts on the $n$-fold product $X^n$ by permuting coordinates,
and the quotient by this action is the \emph{multiset space} 
$\mult_n(X)$. 

If $\Ib$ is isometric to a compact interval in $\mathbb{R}$, then the space
$\Ib^n$ is an $n$-dimensional cube and $\mult_n(\Ib)$ is a 
metric simplex known as a standard $n$-dimensional orthoscheme. 
The face poset of $\mult_n(\Ib)$
is the poset of \emph{linear compositions of $n$},
which is denoted $\comp_n(\closedint)$ and is isomorphic to the
Boolean lattice with its minimum element removed.

In this article, we are interested in the multiset space $\mult_n(\Qb)$,
where $\Qb$ is isometric to a closed rectangle in the complex plane. Our first 
result is to describe a cell structure for $\mult_n(\Qb)$, which is
\emph{bi-simplicial} in the sense that every cell is a product of 
two simplices, as well as an analogous poset of \emph{rectangular compositions} $\comp_n(\closedsquare)$
which records the incidence of faces.

\begin{mainthm}[Theorem~\ref{thm:rectangular-multiset-face-poset}]\label{mainthm:cell-structure}
    Let $\Qb$ be a closed rectangle and let $n$ be a positive integer.
    Then the multiset space $\mult_n(\Qb)$ is isometric to a bi-simplicial
    piecewise Euclidean cell complex with face poset equal to $\comp_n(\closedsquare)$.
\end{mainthm}

For our second main theorem, we require two definitions. First, 
the \emph{dual graph} of a cell complex is formed by placing a vertex in the
interior of each top-dimensional cell, then connecting two vertices if their
corresponding cells share a face of codimension 1. Next, if $S$ is the
standard generating set of adjacent transpositions 
$\{\sigma_1,\ldots,\sigma_{n-1}\}$ for the symmetric group $\sym_n$,
then we define $\Gamma^{LR}(\sym_n,S)$ to be the graph obtained by overlaying
the left and right Cayley graphs of $\sym_n$ with respect to $S$. In other
words, the vertex set of $\Gamma^{LR}(\sym_n,S)$ is $\sym_n$ and 
there are two types of edges: for all $\pi,\pi' \in \sym_n$,
there is an edge between $\pi$ and $\pi'$ labeled ``$\sigma_i\cdot$'' 
if $\sigma_i \cdot \pi = \pi'$ where $\sigma_i \in S$, and
an edge between $\pi$ and $\pi'$ labeled ``$\cdot\sigma_j$''
if $\pi \cdot \sigma_j = \pi'$ where $\sigma_j \in S$.
Note in particular that this graph is not simple, as some pairs of
vertices have two edges between them.
We can now articulate the dual graph for 
$\mult_n(\Qb)$ as follows.

\begin{mainthm}[Theorem~\ref{thm:dual-graph}]\label{mainthm:dual-graph}
    The dual graph of the bi-simplicial cell structure on $\mult_n(\Qb)$ 
    is isomorphic to $\Gamma^{LR}(\sym_n,S)$.
\end{mainthm}

Part of our interest in the space of rectangular multisets comes from
its relationship to spaces of complex polynomials. In particular, if
$\poly_{n+1}^{mc}(\Qb)$ denotes the space of monic degree-$(n+1)$
complex polynomials with roots centered at the origin and critical
values confined to the closed rectangle $\Qb$, then
$\poly_{n+1}^{mc}(\Qb)$ can be endowed with a bi-simplicial cell
structure called the ``branched rectangle complex''
\cite{gcp2}. Moreover, the function which sends each complex
polynomial to its multiset of critical values (known as the
\emph{Lyashko--Looijenga map}---see \cite{LaZv04} for a reference) is
a cellular stratified covering map from $\poly_{n+1}^{mc}(\Qb)$ to
$\mult_n(\Qb)$. In this sense, the cell structure defined in
Theorem~\ref{mainthm:cell-structure} gives rise to the cell structure
on the space of complex polynomials with critical values confined to
$\Qb$, and as shown in \cite{gcp2}, there is a quotient of this cell
complex by face identifications which is homeomorphic to the space of
monic polynomials with $n+1$ \emph{distinct} centered roots, a
classifying space for the $(n+1)$-strand braid group.

The remainder of the article is structured as
follows. Section~\ref{sec:perm-orth} describes our conventions and
definitions for permutations and orthoschemes.  In
Section~\ref{sec:pts-interval}, we examine the space of linear
multisets and define the corresponding poset of linear
compositions. We define the analogous space of rectangular multisets
and the poset of rectangular compositions in
Section~\ref{sec:pt-rect}, where we also provide the proofs of
Theorem~\ref{mainthm:cell-structure} and
Theorem~\ref{mainthm:dual-graph}. Finally, we provide some additional
ways to illustrate bi-orthoschemes via spines in
Section~\ref{sec:spines}.

\section{Permutations and orthoschemes}\label{sec:perm-orth}

In this section we give our conventions for permutations and review
some basic facts about orthoschemes. Throughout this article, let
$[n]$ denote the set $\{1,\ldots,n\}$ where $n$ is a positive integer.

\begin{defn}[Permutations]\label{def:perms}
  Let $S$ be a finite linearly ordered set, which we identify with
  $[n]$. Each permutation $\pi$ in the symmetric group $\sym_n$ can be
  viewed as a bijection between two copies of $S$: one on the left and
  one on the right. Moreover, $\pi$ can be represented as a product of
  disjoint cycles $(a_1\ a_2\ \cdots\ a_k)$, each of which indicates
  that the element $a_i$ in the left copy of $S$ is sent to $a_{i+1}$
  in the right copy (with indices evaluated mod $k$). There are then
  two natural actions of a permutation $\pi \in \sym_n$ on $S$; if
  $(a\ b\ c)$ is part of the disjoint cycle representation of $\pi$
  then the left action gives $a = \pi \cdot b$ (since $b$ on the right
  is sent to $a$ on the left) and the right action gives $b\cdot \pi =
  c$ (since $b$ on the left is sent to $c$ on the right).  That is,
  the left action uses the right copy of $S$ as its domain, and the
  right action uses the left copy of $S$ as its domain.  Finally, if
  $\sigma_i$ is the transposition $(i\ i+1)$ for $i \in [n-1]$, then
  $\{\sigma_1,\ldots,\sigma_{n-1}\}$ forms the standard generating set
  for $\sym_n$.
\end{defn}

Permutations can also be represented as matrices of zeros and ones.

\begin{defn}[Permutation matrices]\label{def:perm-mat}
  A permutation $\pi$ of $[n]$ can be represented as an $n\times n$
  \emph{permutation matrix} $M_\pi = (m_{ij})$ by setting
  $m_{ij} = 1$ if $i\cdot\pi = j$ and $m_{ij} = 0$ otherwise.
  Permutation composition corresponds to matrix multiplication in the
  sense that if $\alpha,\beta\in \sym_n$, then the permutation
  $\alpha\beta$ is represented by the matrix $M_\alpha M_\beta$. 
  In particular, note that the matrix product $M_{\sigma_i} M_\pi$ 
  is obtained from $M_\pi$ by swapping rows $i$ and $i+1$,
  whereas $M_\pi M_{\sigma_i}$ is obtained from $M_\pi$ by swapping
  columns $i$ and $i+1$. As a convenient visual shorthand, we can
  encode a permutation $\pi$ using an $n\times n$ grid in which dots 
  represent the positions of $1$'s in $M_\pi$---see 
  Figure~\ref{fig:dual-graph}. 
\end{defn}

Next, we define the metric simplices known as orthoschemes and
describe their appearance in $n$-dimensional cubes.

\begin{defn}[Orthoschemes and spines]\label{def:orthoscheme}
  Let $\bv_1,\ldots,\bv_n$ be an orthogonal set of vectors in
  Euclidean space. Fix a point $p_0$, and for each $i \in [n]$, define
  the point $p_i = p_0 + \bv_1 + \bv_2 + \cdots + \bv_i$. Then the
  convex hull of the points $p_0,\ldots,p_n$ is an
  \emph{$n$-dimensional orthoscheme}. An orthoscheme is said to be
  \emph{standard} if the vectors $\bv_1,\ldots,\bv_n$ all have equal
  length. Finally, the \emph{spine} of an orthoscheme is the subgraph
  of its 1-skeleton which consists of vertices $p_0,\ldots,p_n$ and
  edges $e_1,\ldots,e_n$ such that for all $i\in [n]$, $e_i$ connects
  $p_{i-1}$ to $p_i$.
\end{defn}

\begin{defn}[Cubes]\label{def:cubes}
  Let $\Ib = [x_\ell, x_r]$ be a closed interval in $\mathbb{R}$ with
  $x_\ell \neq x_r$. The $n$-cube $\Ib^n \subset \mathbb{R}^n$
  inherits a simplicial cell structure from its intersection with the
  real braid arrangement, which consists of the $\binom{n}{2}$
  hyperplanes described by equations of the form $x_i = x_j$ for
  distinct $i,j\in [n]$.  The $n!$ top-dimensional cells are standard
  $n$-dimensional orthoschemes.  The linear order on $\Ib$ extends to
  a coordinate-by-coordinate partial order on $\Ib^n$. Under this
  partial order, the vertices of each simplex are totally ordered, and
  these \emph{ordered simplices} turn $\Ib^n$ into an \emph{ordered
  simplicial complex}.
\end{defn}

\section{Linear multisets and compositions}\label{sec:pts-interval}

In this section we review the cell structure for the multiset
space of a compact interval (isometric to a standard orthoscheme) 
and its associated face poset of linear compositions (isomorphic
to a truncated Boolean lattice). For more details on the content
of this section, see \cite{dm-cnc}.

\begin{defn}[Multiset spaces]
  Let $X$ be a topological space and let $n$ be a positive integer.
  As described in the introduction, the symmetric group $\sym_n$ acts
  on $X^n$ by permuting coordinates, and the quotient is the
  \emph{multiset space} $\mult_n(X)$. Each point in this space is a
  \emph{multiset} which can be represented either formally as a
  function $\bm \colon X \to \mathbb{N}$ such that $\sum_{x \in X}
  \bm(x) = n$ (necessitating that $\bm(x) = 0$ for all but finitely
  many $x\in X$) or informally as a collection of $n$ points in $X$
  where multiplicity is permitted. Note that the multiset space
  $\mult_n(X)$ is distinct from the related configuration space
  $\conf_n(X)$ (in which points are subsets of $X$ with $n$ elements)
  and subset space $\exp_n(X)$ (in which points are subsets of $X$
  with at most $n$ elements); see e.g. \cite{tuffley02} for a
  comparison of the latter two cases.
\end{defn}

\begin{defn}[Linear multisets]\label{def:multiset-line}
  Let $\Ib = [x_\ell, x_r]$ as in Definition~\ref{def:cubes}.  The
  symmetric group $\sym_n$ acts on $\Ib^n$ by permuting coordinates in
  two ways: for each $\sigma \in \sym_n$ and $\bx =
  (x_1,x_2,\ldots,x_n)$ in $\Ib^n$, define the left action
  $\sigma\cdot\bx = (x_{\sigma\cdot 1}, x_{\sigma\cdot 2},\ldots,
  x_{\sigma\cdot n})$ and the right action $\bx \cdot \sigma =
  (x_{n\cdot\sigma},\ldots,x_{2\cdot\sigma},x_{1\cdot\sigma})$.  In
  both cases, the action preserves the ordering on each simplex, and
  the multiset space $\mult_n(\Ib)$ is a single ordered
  $n$-dimensional orthoscheme.
\end{defn}

Each linear multiset $\bx \in \mult_n(\Ib)$ can be uniquely denoted
with the shorthand notation $\bx = x_\ell^{m_\ell}x_1^{m_1} \cdots
x_k^{m_k} x_r^{m_r}$, where $x_\ell < x_1 < \cdots < x_k < x_r$ in
$\Ib$ and $m_i$ denotes the multiplicity of $x_i$ in $\bx$. Note that
this notation always includes the endpoints of $\Ib$, even if they do
not appear in $\bx$; in other words, $m_\ell$ and $m_r$ may be zero,
while $m_1,\ldots,m_k$ are strictly positive. As for the cell
structure, the simplices in $\mult_n(\Ib)$ admit a natural labeling by
compositions of $n$ in the following sense.

\begin{defn}[Linear compositions]\label{def:linear-compositions}
  Let $n$ be a positive integer and let $k$ be a nonnegative integer.
  A \emph{linear composition of $n$ with length $k+2$} is a row vector
  $\ba = [a_\ell\ a_1\ \cdots\ a_k\ a_r]$ of integers with entry sum
  $n$ such that $a_\ell = a_0$ and $a_r = a_{k+1}$ are nonnegative,
  and every other $a_i$ is positive for $i\in [k]$.  A composition of
  length $k+2$ can be \emph{merged} at position $i \in [k+1]$ by
  replacing the two entries $a_{i-1}$ and $a_{i}$ with the single
  entry $a_{i-1} + a_{i}$ to obtain a new composition of $n$ with
  length $k+1$.  Let $\comp_n(\closedint)$ denote the set of all
  linear compositions of $n$ together with the partial order $\ba \leq
  \bb$ if there is a sequence of merges which can be applied to $\bb$
  to obtain $\ba$.  Note that if $\ba$ is a vector of length $k+2$
  with nonnegative integer entries which sum to $n$, then $\ba$ is a
  linear composition of $n$ if and only if none of its internal
  (i.e. non-first, non-last) entries are zero.
\end{defn}

Note that $\comp_n(\closedint)$ has a unique maximum element
$[0\ 1\ \cdots\ 1\ 0]$ and each element below it is determined by
performing a subset of the $n+1$ possible merges.  In other words,
$\comp_n(\closedint)$ is isomorphic to $\bool_{n+1}^*$, the Boolean
lattice with its minimum element removed. Although $\bool_{n+1}^*$ has
a more natural definition as the nonempty subsets of $[n+1]$, our
definition for $\comp_n(\closedint)$ provides a direct connection with
the set of multisets in $\Ib$.

Define the surjective map $\comp \colon \mult_n(\Ib) \to \comp_n(\closedint)$ by
\[\comp(x_\ell^{m_\ell}x_1^{m_1} \cdots x_k^{m_k} x_r^{m_r}) = [m_\ell\ m_1\ \cdots\ m_k\ m_r]\]
and recall that the \emph{face poset} of a polyhedral cell complex is
the set of all nonempty faces, partially ordered by inclusion (see
e.g.  \cite{wachs07}). We then have the following proposition.

\begin{prop}[\protect{\cite[Example~4.4]{dm-cnc}}]\label{prop:linear-multiset-face-poset}
  The face poset for $\mult_n(\Ib)$ is isomorphic to
  $\comp_n(\closedint)$.
\end{prop}

\begin{proof}
  The preimage of each linear composition of $n$ under the $\comp$ map
  is an open cell in $\mult_n(\Ib)$, and $\ba \leq \bb$ in
  $\comp_n(\closedint)$ if and only if $\comp^{-1}(\ba)$ is a subset
  of the closure of $\comp^{-1}(\bb)$. Thus, each face of
  $\mult_n(\Ib)$ is uniquely labeled by a linear composition of $n$,
  and the partial order of merging in $\comp_n(\closedint)$
  corresponds exactly to the inclusion of faces in $\mult_n(\Ib)$.
\end{proof}

See Figure~\ref{fig:comp-3-orthoscheme} for an illustration of the
$3$-dimensional orthoscheme $\mult_3(\Ib)$ with cells labeled by the
linear compositions of $3$. To close this section, we note some useful
ways of visualizing the space of linear multisets and the
corresponding linear compositions.

\begin{figure}
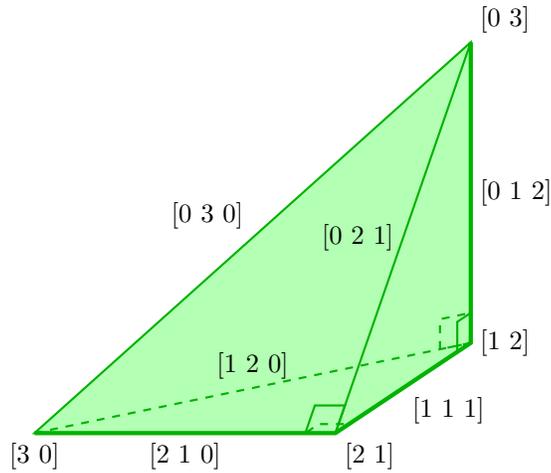

    \centering
    \includestandalone{fig-comp-3-orthoscheme}
    \caption{The space $\mult_3(\Ib)$ with labels on the $1$-skeleton.
    The unique $3$-cell is labeled $[0\ 1\ 1\ 1\ 0]$ and the four
    $2$-cells have labels $[1\ 1\ 1\ 0]$ (bottom), $[0\ 2\ 1\ 0]$ (front), 
    $[0\ 1\ 2\ 0]$ (back), and $[0\ 1\ 1\ 1]$ (right). The spine of
    the orthoscheme is illustrated with bold edges.}
    \label{fig:comp-3-orthoscheme}
\end{figure}

\begin{rem}[Illustrating orthoschemes]
    \label{rem:spine-orthoscheme}
    Given a multiset $\bx \in \mult_n(\Ib)$, it is easy to read off
    not only the linear composition $\comp(\bx)$, but also the linear compositions
    which label the vertices of the corresponding orthoscheme: each vertex
    label is obtained by preserving a single gap between entries and merging
    at all other positions in $\comp(\bx)$. For example, if 
    $\bx = x_\ell^3 x_1^4 x_2^1 x_3^2 x_r^1 \in \mult_{11}(\Ib)$,
    then $\bx$ lies in the interior of the $3$-dimensional orthoscheme labeled by the 
    linear composition $\comp(\bx) = [3\ 4\ 1\ 2\ 1]$, and the spine of this orthoscheme
    goes through the vertices $[3\ 8]$, $[7\ 4]$, $[8\ 3]$ and $[10\ 1]$, in order.
    Moreover, the multiplicities of $x_1$, $x_2$ and $x_3$ tell us that the
    edges connecting these vertices are of lengths $\sqrt{4}$, $\sqrt{1}$, and $\sqrt{2}$
    respectively.
    See Figure~\ref{fig:spine-orthoscheme} for an illustration which overlays the
    spine of the orthoscheme labeled by $\comp(\bx)$ on the multiset $\bx$.
\end{rem}

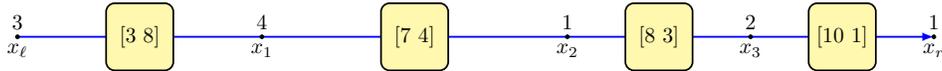
\begin{figure}
    \centering
    \begin{tikzpicture}

\tikzstyle{BlueLine}=[line width=0.3mm,color=blue,text=black]
\tikzstyle{BluePoly}=[BlueLine,fill=blue!20]
\tikzstyle{RedLine}=[line width=0.3mm,color=red,text=black]
\tikzstyle{RedPoly}=[RedLine,fill=red!20]
\tikzstyle{GreenLine}=[thick,color=black!30!green,text=black]
\tikzstyle{GreenPoly}=[thick,color=green!50!black,fill=green!30,join=bevel]
\tikzstyle{PurpleLine}=[line width=0.3mm,color=red!60!blue!80,text=black]
\tikzstyle{PurplePoly}=[PurpleLine,fill=red!40!blue!50]
\tikzstyle{OrangeLine}=[thick,color=orange]
\tikzstyle{GrayLine}=[thick,color=black!50!gray]
\tikzstyle{GrayPoly}=[GrayLine,fill=gray!20]
\tikzstyle{dot}=[shape=circle,draw,color=black,fill=black,inner sep=1.5pt]
\tikzstyle{opendot}=[shape=circle,draw,color=black,fill=white,inner sep=1.5pt]
\tikzstyle{bigdot}=[dot,inner sep=2pt]
\tikzstyle{littledot}=[dot,inner sep=0.75pt]
\tikzstyle{littleopendot}=[opendot,inner sep=0.75pt]
\tikzstyle{smalldot}=[dot,inner sep=1pt]
\tikzstyle{disk}=[thick,shape=circle,draw,color=black,fill=yellow!10]
\tikzstyle{lightdisk}=[thick,shape=circle,draw,color=black!40,fill=yellow!20]
\tikzstyle{plate}=[thick,shape=rectangle,draw,color=black,fill=yellow!40,rounded corners,minimum size=1.1cm]

\coordinate (xl) at (-10,0);
\coordinate (x1) at (-6,0);
\coordinate (x2) at (-1,0);
\coordinate (x3) at (2,0);
\coordinate (xr) at (5,0);

\draw[BlueLine, -latex] (xl) -> (xr);

\foreach \v in {xl, x1, x2, x3, xr} {
    \filldraw[dot,color=black] (\v) circle (.3mm);
}

\node[anchor=north] () at (xl) {$x_\ell$};
\node[anchor=north] () at (x1) {$x_1$};
\node[anchor=north] () at (x2) {$x_2$};
\node[anchor=north] () at (x3) {$x_3$};
\node[anchor=north] () at (xr) {$x_r$};

\node[plate] () at ($(xl)!0.5!(x1)$) {$[3\ 8]$};
\node[plate] () at ($(x1)!0.5!(x2)$) {$[7\ 4]$};
\node[plate] () at ($(x2)!0.5!(x3)$) {$[8\ 3]$};
\node[plate] () at ($(x3)!0.5!(xr)$) {$[10\ 1]$};

\node[anchor=south] () at (xl) {$3$};
\node[anchor=south] () at (x1) {$4$};
\node[anchor=south] () at (x2) {$1$};
\node[anchor=south] () at (x3) {$2$};
\node[anchor=south] () at (xr) {$1$};

\end{tikzpicture}
    \caption{The multiset $\bx \in \mult_{11}(\Ib)$ defined in Remark~\ref{rem:spine-orthoscheme}, 
    with the linear compositions labeling the spine of the $3$-dimensional orthoscheme overlaid.}
    \label{fig:spine-orthoscheme}
\end{figure}

\section{Rectangular multisets and compositions}\label{sec:pt-rect}

In this section we introduce a polyhedral cell structure for 
the space of $n$-multisets in a rectangle. First, we establish
a coloring mnemonic for product spaces.

\begin{rem}[Iconic colors]\label{rem:colors}
  When dealing with spaces such as rectangles which split as a direct
  product into horizontal and vertical factors, we adopt the following
  color convention. Things constructed from the
  horizontal portion are colored \emph{blue}, like the flat surface of
  water, and things constructed from the vertical portion are colored
  \emph{red}, like the rising flames of a fire.
\end{rem}

Next, we introduce some terminology for products of simplices and
then define rectangular multisets, the main object of study in this section.

\begin{defn}[Bi-simplices]
    \label{def:bisimplices}
    A \emph{bi-simplex} is a polytope which can be expressed as the product 
    of two simplices. When each simplex is equipped with the metric of
    an orthoscheme, we refer to the product as \emph{bi-orthoscheme} instead.
    Moreover, the \emph{spine} of a bi-orthoscheme is defined to be the
    subcomplex of its $2$-skeleton which is formed by taking the product of
    the spines of the factor orthoschemes.
\end{defn}

\begin{defn}[Rectangular multisets]
    \label{def:multiset-rect}
    Let $\Ib = [x_\ell, x_r]$ and $\Jb = [y_b,y_t]$ be closed intervals in
    $\mathbb{R}$ with $x_\ell \neq x_r$ and $y_b \neq y_t$, 
    and define $\Qb = \Ib \times \Jb$, viewed as the \emph{coordinate rectangle} 
    $\Qb = \Ib +\Jb i = \{ z \mid \Re(z) \in \Ib, \Im(z) \in \Jb \}$ 
    in $\mathbb{C}$. Following the convention established in 
    Remark~\ref{rem:colors}, we think of $\Ib$ as blue and $\Jb$ as red.
    The $n$-fold product $\Qb^n$ can then be viewed as the product of 
    a blue $n$-cube $\Ib^n$ and a red $n$-cube $\Jb^n$, each of which
    has a simplicial cell structure from Definition~\ref{def:multiset-line}.
    The result is a bi-simplicial cell structure on $\Qb^n$ with
    $(n!)^2$ top-dimensional cells, each of which is the product of two 
    $n$-dimensional orthoschemes (one blue, one red). 
    The action of $\sym_n$ on $\Qb^n$ by
    permuting coordinates can be viewed as the diagonal action on 
    $\Ib^n \times \Jb^n$ which restricts to the action from
    Definition~\ref{def:multiset-line} on each coordinate, and the quotient
    $\mult_n(\Qb)$ is a cell complex with $n!$ top-dimensional cells.
\end{defn}

Viewing $\Qb$ as a subset of the complex plane, the projection maps
onto real and imaginary parts give us projections 
$\Re \colon \Qb \to \Ib$ and
$\Im \colon \Qb \to \Jb$, and they naturally extend
to projections $\Re \colon \mult_n(\Qb) \to \mult_n(\Ib)$
and $\Im \colon \mult_n(\Qb) \to \mult_n(\Jb)$. Each
multiset $\bz \in \mult_n(\Qb)$ uniquely determines 
the pair $(\bx,\by) = (\Re(\bz),\Im(\bz))$, and we denote
$\bx$ and $\by$ as 
\begin{align*}
    \bx &= x_\ell^{c_\ell}x_1^{c_1}\cdots x_h^{c_h}x_r^{c_r} 
    &(x_\ell < x_1 < \cdots < x_h < x_r) \\
    \by &= y_b^{d_b}y_1^{d_1}\cdots y_k^{d_k}y_t^{d_t}
    &(y_b < y_1 < \cdots < y_k < y_t).
\end{align*}
As noted in Section~\ref{sec:pts-interval}, the endpoints $x_\ell$,
$x_r$, $y_b$ and $y_t$ appear in the notation for $\bx$ and $\by$
even if they are not elements of those multisets; the multiplicities for those
entries may be zero.
Moreover, the map $\mult_n(\Qb) \to \mult_n(\Ib) \times \mult_n(\Jb)$
given by $\bz \mapsto (\bx,\by)$ is a cellular surjection, and each point in
the interior of the bi-orthoscheme $\mult_n(\Ib) \times \mult_n(\Jb)$ has exactly $n!$ 
preimages in $\mult_n(\Qb)$. More generally, the number of preimages 
can be enumerated by counting certain types of matrices with specified row 
sums and column sums.

As in the case of linear multisets, the higher-dimensional cells of $\mult_n(\Qb)$
can be understood using a convenient combinatorial tool. To smooth the transition from 
multisets in planar rectangles to integers in matrices, we need to navigate a 
discrepancy between planar conventions and matrix conventions.

\begin{rem}[Reconciling orientations]\label{rem:rotate}
    The standard planar coordinate system with points $(x,y)$ has the positive $x$-axis 
    pointing to the right and the positive $y$-axis pointing upwards.  The standard discrete 
    matrix coordinate system has entries $(i,j)$ with the row index $i$ 
    increasing as one descends vertically, and the column index $j$ increasing as 
    one moves to the right. To reconcile the differences between these two conventions, 
    we choose to rotate the plane clockwise by a quarter-turn so that it appears in 
    ``matrix orientation'' (with the positive $x$-axis pointing downward and the positive 
    $y$-axis pointing to the right). See the lefthand side of Figure~\ref{fig:mult-rect-projections}, 
    where the planar rectangle $\Qb$ is in shown in ``matrix orientation''.  In fact, 
    matrix orientation is used consistently in all of the figures.  Our color conventions 
    (Remark~\ref{rem:colors}) remain based on ``planar orientation''.
\end{rem}    

\begin{rem}[Multiplicity matrix]\label{rem:mult-matrix}
  For each rectangular multiset $\bz$ of $n$ points in the rectangle $\Qb$, 
  we define a \emph{multiplicity matrix}  that records the multiplicities of points
  in $\bz$, with entry sum $n$.  We use the natural order of the reals in the 
  intervals $\Ib = [x_\ell,x_r]$ and $\Jb = [y_b,y_t]$ to index the rows and columns
  respectively.  When the rectangle $\Qb$ is viewed in matrix orientation (Remark~\ref{rem:rotate}),
  the multiset in $\Qb$ visually aligns with the multiplicity matrix. Figure~\ref{fig:mult-rect-projections} 
  illustrates the process: the multiset $\bz$ contains a point of multiplicity 
  $1$ at the point $(x_1,y_t)$ in $\Qb$, so the corresponding matrix
  has a $1$ in the $(2,4)$-entry. See Definition~\ref{def:mult-mat}
  for a formal description.
\end{rem}

\begin{defn}[Rectangular compositions]
    \label{def:rect-compositions}
    Let $n$ be a positive integer and let $h$ and $k$ be nonnegative integers. 
    Given an $(h+2)\times(k+2)$ matrix
    \[
        \ba = \begin{bmatrix}
            a_{\textcolor{orange}{\ell},\textcolor{blue}{b}} & a_{\textcolor{orange}{\ell},1} & 
            \cdots & a_{\textcolor{orange}{\ell},k} & a_{\textcolor{orange}{\ell},\textcolor{cyan}{t}} \\
            a_{1,\textcolor{blue}{b}} & a_{1,1} & \cdots & a_{1,k} & a_{1,\textcolor{cyan}{t}} \\
            \vdots & \vdots & \ddots & \vdots & \vdots \\
            a_{h,\textcolor{blue}{b}} & a_{h,1} & \cdots & a_{h,k} & a_{h,\textcolor{cyan}{t}} \\
            a_{\textcolor{red}{r},\textcolor{blue}{b}} & a_{\textcolor{red}{r},1} & 
            \cdots & a_{\textcolor{red}{r},k} & a_{\textcolor{red}{r},\textcolor{cyan}{t}}
        \end{bmatrix}
    \]  
    of nonnegative integers, let $c_\ell, c_1, \ldots, c_h, c_r$ denote the
    row sums of $\ba$ and let $d_b, d_1, \ldots, d_k, d_t$
    denote the column sums of $\ba$. We say that $\ba$ is a
    \emph{rectangular composition of $n$ with shape $(h+2) \times (k+2)$} if
    \begin{enumerate}
        \item $c_i > 0$ for all $i \in [h]$ ($c_\ell$ and $c_r$ can be zero);
        \item $d_j > 0$ for all $j \in [k]$ ($d_t$ and $d_b$ can be zero);
        \item $c_\ell + c_1 + \cdots + c_h + c_r = n$.
    \end{enumerate}
    A rectangular composition can be \emph{row-merged} (resp. \emph{column-merged})
    to obtain a new rectangular composition of $n$ with shape $(h+1)\times(k+2)$
    (resp. $(h+2)\times (k+1)$) by removing two adjacent rows (resp. columns) and
    inserting their sum. Let $\comp_n(\closedsquare)$ denote the set of all rectangular
    compositions of $n$ and define a partial order by declaring $\ba \leq \ba'$
    if and only if there is a sequence of row-merges and/or column-merges
    which transforms $\ba'$ into $\ba$. Finally, note that if $\ba$ is an
    $(h+2)\times(k+2)$ matrix with nonnegative integer entries which sum to $n$, 
    then $\ba$ is a rectangular composition of $n$ if and only if it has no internal
    rows or internal columns which sum to zero. 
\end{defn}

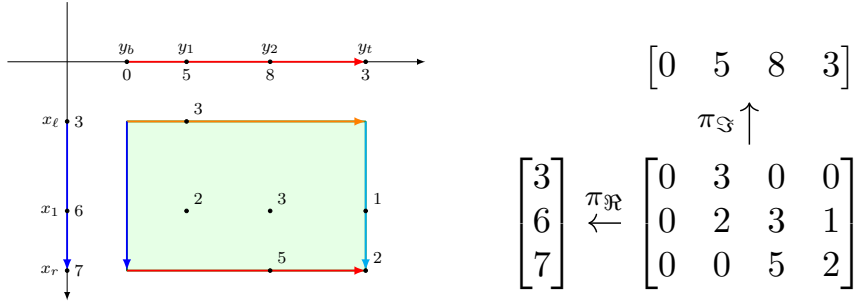
\begin{figure}
    \centering
    \begin{tikzpicture}

\tikzstyle{BlueLine}=[line width=0.3mm,color=blue,text=black]
\tikzstyle{BluePoly}=[BlueLine,fill=blue!20]
\tikzstyle{RedLine}=[line width=0.3mm,color=red,text=black]
\tikzstyle{RedPoly}=[RedLine,fill=red!20]
\tikzstyle{GreenLine}=[thick,color=black!30!green,text=black]
\tikzstyle{GreenPoly}=[thick,color=green!50!black,fill=green!30,join=bevel]
\tikzstyle{PurpleLine}=[line width=0.3mm,color=red!60!blue!80,text=black]
\tikzstyle{PurplePoly}=[PurpleLine,fill=red!40!blue!50]
\tikzstyle{OrangeLine}=[thick,color=orange]
\tikzstyle{GrayLine}=[thick,color=black!50!gray]
\tikzstyle{GrayPoly}=[GrayLine,fill=gray!20]
\tikzstyle{dot}=[shape=circle,draw,color=black,fill=black,inner sep=1.5pt]
\tikzstyle{opendot}=[shape=circle,draw,color=black,fill=white,inner sep=1.5pt]
\tikzstyle{bigdot}=[dot,inner sep=2pt]
\tikzstyle{littledot}=[dot,inner sep=0.75pt]
\tikzstyle{littleopendot}=[opendot,inner sep=0.75pt]
\tikzstyle{smalldot}=[dot,inner sep=1pt]
\tikzstyle{disk}=[thick,shape=circle,draw,color=black,fill=yellow!10]
\tikzstyle{lightdisk}=[thick,shape=circle,draw,color=black!40,fill=yellow!20]
\tikzstyle{plate}=[thick,shape=rectangle,draw,color=black,fill=yellow!10,rounded corners,minimum size=1.1cm]


\def\shade{.8}

\definecolor{c00}{rgb}{0,\shade,\shade} 
\definecolor{c60}{rgb}{0,0,\shade}
\definecolor{c120}{rgb}{\shade,0,\shade} 
\definecolor{c180}{rgb}{\shade,0,0}
\definecolor{c240}{rgb}{\shade,\shade,0} 
\definecolor{c300}{rgb}{0,\shade,0}
\definecolor{c360}{rgb}{0,\shade,\shade}

\colorlet{c10}{c60!16.7!c00} \colorlet{c20}{c60!33.3!c00}
\colorlet{c30}{c60!50!c00}   \colorlet{c40}{c60!66.7!c00}
\colorlet{c50}{c60!83.3!c00}

\colorlet{c70}{c120!16.7!c60} \colorlet{c80}{c120!33.3!c60}
\colorlet{c90}{c120!50!c60}   \colorlet{c100}{c120!66.7!c60}
\colorlet{c110}{c120!83.3!c60}

\colorlet{c130}{c180!16.7!c120} \colorlet{c140}{c180!33.3!c120}
\colorlet{c150}{c180!50!c120}   \colorlet{c160}{c180!66.7!c120}
\colorlet{c170}{c180!83.3!c120}

\colorlet{c190}{c240!16.7!c180} \colorlet{c200}{c240!33.3!c180}
\colorlet{c210}{c240!50!c180}   \colorlet{c220}{c240!66.7!c180}
\colorlet{c230}{c240!83.3!c180}

\colorlet{c250}{c300!16.7!c240} \colorlet{c260}{c300!33.3!c240}
\colorlet{c270}{c300!50!c240}   \colorlet{c280}{c300!66.7!c240}
\colorlet{c290}{c300!83.3!c240}

\colorlet{c310}{c360!16.7!c300} \colorlet{c320}{c360!33.3!c300}
\colorlet{c330}{c360!50!c300}   \colorlet{c340}{c360!66.7!c300}
\colorlet{c350}{c360!83.3!c300}

\colorlet{ca}{c00}
\colorlet{cb}{c60}
\colorlet{cc}{c180}
\colorlet{cd}{c300}

\colorlet{bckgrd}{green!10!white}
\colorlet{cA}{black!10!white}
\colorlet{cB}{black!20!white}
\colorlet{cC}{black!40!white}
\colorlet{cD}{black!60!white}
\colorlet{cE}{black!80!white}

\def\yb{1}
\def\yt{5}
\def\ym{3}
\def\yma{2}
\def\ymb{3.4}
\def\xl{1}
\def\xr{3.5}
\def\xm{2.25}
\def\xma{2.5}
\coordinate (xl) at (0,-\xl);
\coordinate (xr) at (0,-\xr);
\coordinate (xm) at (0,-\xm);
\coordinate (xma) at (0,-\xma);
\coordinate (yb) at (\yb,0);
\coordinate (yt) at (\yt,0);
\coordinate (ym) at (\ym,0);
\coordinate (yma) at (\yma,0);
\coordinate (ymb) at (\ymb,0);
\coordinate (v1) at (\yt,-\xr);
\coordinate (v2) at (\yt,-\xl);
\coordinate (v3) at (\yb,-\xl);
\coordinate (v4) at (\yb,-\xr);
\coordinate (t) at (\yt,-\xm);
\coordinate (l) at (\ym,-\xl);
\coordinate (b) at (\yb,-\xm);
\coordinate (r) at (\ym,-\xr);

\coordinate (z1) at (\yma,-\xma);
\coordinate (z2) at (\yma,-\xl);
\coordinate (z3) at (\ymb,-\xr);
\coordinate (z4) at (\ymb,-\xma);
\coordinate (z5) at (\yt,-\xr);
\coordinate (z6) at (\yt,-\xma);

\draw[-latex] (0,1)--(0,-4);
\draw[-latex] (-1,0)--(6,0);

\filldraw[GreenPoly, fill = bckgrd] (v1)--(v2)--(v3)--(v4)--cycle;

\draw[BlueLine,-latex] (xl)--(xr);

\draw[RedLine,-latex] (yb)--(yt);

\draw[RedLine,-latex,color=orange] (v3)--(v2);
\draw[RedLine,-latex,color=red] (v4)--(v1);
\draw[BlueLine,-latex,color=cyan] (v2)--(v1);
\draw[BlueLine,-latex,color=blue] (v3)--(v4);

\foreach \p in {xl,xr,yb,yt,xma,ymb,yma,z1,z2,z3,z4,z5,z6}{
    \filldraw[dot,color=black] (\p) circle (.3mm);
}

\node[anchor=east] at (xl) {\footnotesize $x_\ell$};
\node[anchor=west] at (xl) {\footnotesize $3$};
\node[anchor=east] at (xma) {\footnotesize $x_1$};
\node[anchor=west] at (xma) {\footnotesize $6$};
\node[anchor=south] at (ymb) {\footnotesize $y_2$};
\node[anchor=north] at (ymb) {\footnotesize $8$};
\node[anchor=east] at (xr) {\footnotesize $x_r$};
\node[anchor=west] at (xr) {\footnotesize $7$};
\node[anchor=south] at (yb) {\footnotesize $y_b$};
\node[anchor=north] at (yb) {\footnotesize $0$};
\node[anchor=south] at (yma) {\footnotesize $y_1$};
\node[anchor=north] at (yma) {\footnotesize $5$};
\node[anchor=south] at (yt) {\footnotesize $y_t$};
\node[anchor=north] at (yt) {\footnotesize $3$};

\node[anchor=south west] at (z1) {\footnotesize $2$};
\node[anchor=south west] at (z2) {\footnotesize $3$};
\node[anchor=south west] at (z3) {\footnotesize $5$};
\node[anchor=south west] at (z4) {\footnotesize $3$};
\node[anchor=south west] at (z5) {\footnotesize $2$};
\node[anchor=south west] at (z6) {\footnotesize $1$};

\end{tikzpicture}
    \hspace*{2em}
    \begin{tikzpicture}

    \node (c) at (0,0) {
        $\begin{bmatrix} 
            0 & 3 & 0 & 0  \\ 
            0 & 2 & 3 & 1  \\ 
            0 & 0 & 5 & 2  
        \end{bmatrix}$};
    
    \node (a) at (-2,0) {$\begin{bmatrix} 3 \\ 6 \\ 7 \end{bmatrix}$};

    \node (b) at (0,1.5) {$\begin{bmatrix} 0 & 5 & 8 & 3 \end{bmatrix}$};

    \draw[->] (c) -- (a) node[midway,above] {\footnotesize $\pi_\Re$};
    \draw[->] (c) -- (b) node[midway,left] {\footnotesize $\pi_\Im$};
\end{tikzpicture}
    \caption{On the left, an element $\bz \in \mult_{16}(\Qb)$ along with its
    projections $\bx = \Re(\bz) = x_\ell^3 x_1^6 x_r^7$ and
    $\by = \Im(\bz) = y_b^0 y_1^5 y_2^8 y_t^3$. 
    The labels on each point in $\Qb$ correspond to the multiplicity of 
    that point in $\bz$. On the right, the 
    rectangular composition $\comp(\bz)$ along with the linear compositions
    $\comp(\bx)$ and $\comp(\by)$.}
    \label{fig:mult-rect-projections}
\end{figure}

The multiplicities of an $n$-multiset in a rectangle $\Qb$ determine a rectangular
composition of $n$ in the following sense.

\begin{defn}[Multiplicity matrix]\label{def:mult-mat}
For a multiset $\bz \in \mult_n(\Qb)$, let 
$(\bx,\by) = (\Re(\bz),\Im(\bz))$ as above, and for each 
$i \in \{\ell,1,\ldots,h,r\}$ and $j \in \{b,1,\ldots,k,t\}$, define $m_{i,j}\geq 0$ 
to be the multiplicity of $(x_i,y_j)$ in $\bz$. 
Then we define
\[
    \comp(\bz) = \begin{bmatrix}
        m_{\textcolor{orange}{\ell},\textcolor{blue}{b}} & m_{\textcolor{orange}{\ell},1} & 
        \cdots & m_{\textcolor{orange}{\ell},k} & m_{\textcolor{orange}{\ell},\textcolor{cyan}{t}} \\
        m_{1,\textcolor{blue}{b}} & m_{1,1} & \cdots & m_{1,k} & m_{1,\textcolor{cyan}{t}} \\
        \vdots & \vdots & \ddots & \vdots & \vdots \\
        m_{h,\textcolor{blue}{b}} & m_{h,1} & \cdots & m_{h,k} & m_{h,\textcolor{cyan}{t}} \\
        m_{\textcolor{red}{r},\textcolor{blue}{b}} & m_{\textcolor{red}{r},1} & 
        \cdots & m_{\textcolor{red}{r},k} & m_{\textcolor{red}{r},\textcolor{cyan}{t}}
    \end{bmatrix}
\] 
and observe that this defines a surjective map 
$\comp\colon \mult_n(\Qb) \to \comp_n(\closedsquare)$. 
\end{defn}

The $n!$ maximal elements of $\comp_n(\closedsquare)$ are all of shape $(n+2) \times (n+2)$
and can be viewed as $n\times n$ permutation matrices which have been padded with 
zeros along the outside. The minimal elements of $\comp_n(\closedsquare)$ are $2\times 2$ matrices
with entries in $\{0,1,\ldots,n\}$ and sum $n$, and these are counted by the 
\emph{tetrahedral numbers} $\binom{n+3}{3}$. To see why, note that these $2\times 2$
matrices are in one-to-one correspondence with the integer lattice points in $\R^4$ which
lie on the tetrahedron formed by the intersection of the hyperplane 
$x_1 + x_2 + x_3 + x_4 = n$ with the positive orthant.

Each rectangular composition determines a pair of linear compositions, and these are compatible
with the projection maps $\Re: \mult_n(\Qb) \to \mult_n(\Ib)$ and $\Im: \mult_n(\Qb) \to \mult_n(\Jb)$
described above.

\begin{defn}[Composition projections]\label{def:comp-proj}
    There are two natural projection maps $\comp_n(\closedsquare) \to \comp_n(\closedint)$.
    If $\bz \in \mult_n(\Qb)$ with $\ba = \comp(\bz)$ the corresponding
    rectangular composition of $n$, then we define 
    $\pi_{\Re}(\ba) = \pi_{\Re}(\comp(\bz)) = \comp(\Re(\bz))$ and
    $\pi_{\Im}(\ba) = \pi_{\Im}(\comp(\bz)) = \comp(\Im(\bz))$. In other words, if 
    $\ba$ has row sums $c_\ell, c_1, \ldots, c_h, c_r$ and column sums 
    $d_b, d_1, \ldots, d_k, d_t$, then $\pi_{\Re}(\ba) = [c_\ell\ c_1\ \cdots\ c_h\ c_r]$
    and $\pi_{\Im}(\ba) = [d_b\ d_1\ \cdots\ d_k\ d_t]$. Moreover, recall that if $\ba$ is 
    an $(h+2)\times (k+2)$ matrix with nonnegative integer entries that sum to $n$, 
    then $\ba$ is a rectangular composition if and only if it has no internal rows 
    or internal columns which sum to zero; this is equivalent to saying that the
    projections $\pi_{\Re}(\ba)$ and $\pi_{\Im}(\ba)$ have no internal entries of zero, which
    is in turn equivalent to $\pi_{\Re}(\ba)$ and $\pi_{\Im}(\ba)$ being linear compositions of $n$.
\end{defn}

Analogous to the case of linear compositions, we have the following theorem.

\begin{thm}[Theorem~\ref{mainthm:cell-structure}]\label{thm:rectangular-multiset-face-poset}
    $\comp_n(\closedsquare)$ is the face poset for $\mult_n(\Qb)$.
\end{thm}

\begin{proof}
    As described in Definition~\ref{def:multiset-rect}, each open cell
    in $\mult_n(\Qb)$ can be viewed as the product of an open blue orthoscheme
    and an open red orthoscheme. If $\bz \in \mult_n(\Qb)$, then the projections
    $\Re(\bz)$ and $\Im(\bz)$ label points in the blue and red factor orthoschemes,
    respectively. We can then deform $\bz$ by fixing $\Im(\bz)$ and allowing $\Re(\bz)$
    to vary without intersections; these deformed points have the same label in the
    red factor, but produce all possible points in the blue factor. Analogously,
    fixing $\Re(\bz)$ while varying $\Im(\bz)$ produces all possible points in the 
    red orthoscheme while fixing the point in the blue orthoscheme. Since these
    two deformations can be performed independently, we can obtain all points
    in the open cell containing $\mult_n(\Qb)$ in this manner, and it is clear that
    $\comp(\bz)$ remains constant throughout.  Moreover, we obtain all points in the preimage 
    $\comp^{-1}(\comp(z))$ through these deformations.
    In other words, the preimage
    of each rectangular composition of $n$ under the map $\comp \colon
    \mult_n(\Qb) \to \comp_n(\closedsquare)$ is an open bi-simplex.
    
    Furthermore, $\ba \leq \bb$ in $\comp_n(\closedsquare)$ if and only if
    $\ba$ can be obtained from $\bb$ by a sequence of row-merges and/or
    column-merges. If we consider $\comp^{-1}(\bb)$ as a product of 
    a blue orthoscheme and a red orthoscheme, then each row-merge (resp. column-merge)
    corresponds to replacing the blue orthoscheme (resp. red orthoscheme)
    with one of its facets, so the resulting open bi-simplex $\comp^{-1}(\ba)$
    is a subset of the closure of $\comp^{-1}(\bb)$. Since each cell of
    $\mult_n(\Qb)$ can thus be labeled by a unique rectangular composition
    in a manner which respects the incidence of cells, the proof is complete.
\end{proof}

There are many useful ways of illustrating bi-orthoschemes, some of which
are described in the following remark.

\begin{figure}
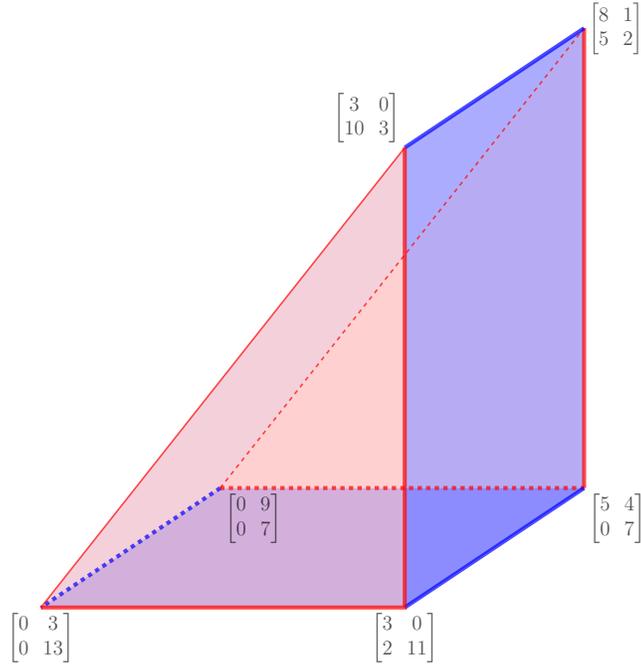

    \centering
    \includestandalone[width=0.7\textwidth]{fig-prism-rect-comp}
    \caption{The bi-orthoscheme labeled by the rectangular composition defined in 
    Remark~\ref{rem:spine-biorthoscheme} (the direct product of a red triangle with a blue
    line segment) with vertex set labeled. The spine of this bi-orthoscheme, consisting of 
    two rectangular faces along with their seven edges and six vertices, is highlighted.}
    \label{fig:prism-rect-comp}
\end{figure}

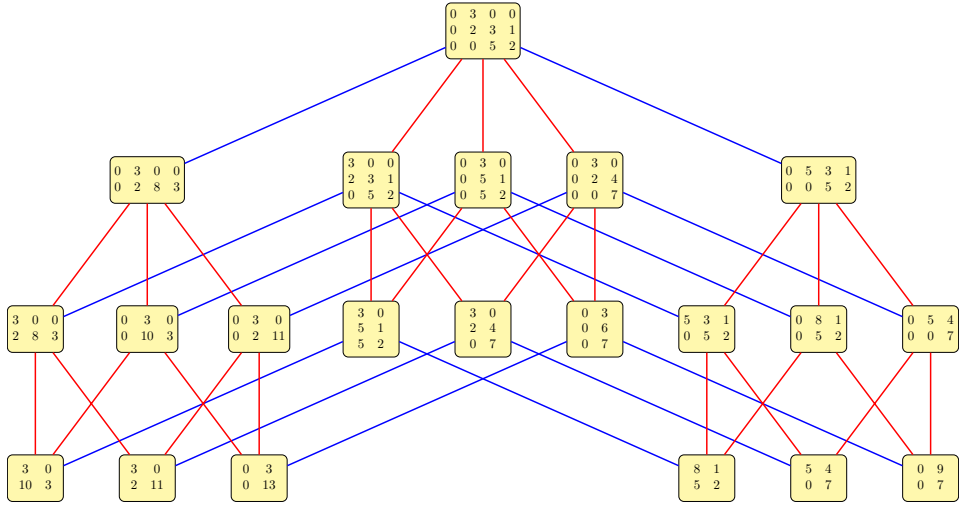
\begin{figure}
    \centering
    \begin{tikzpicture}

\tikzstyle{BlueLine}=[line width=0.4mm,color=blue,text=black]
\tikzstyle{BluePoly}=[BlueLine,fill=blue!20]
\tikzstyle{RedLine}=[line width=0.4mm,color=red,text=black]
\tikzstyle{RedPoly}=[RedLine,fill=red!20]
\tikzstyle{GreenLine}=[thick,color=black!30!green,text=black]
\tikzstyle{GreenPoly}=[thick,color=green!50!black,fill=green!30,join=bevel]
\tikzstyle{PurpleLine}=[line width=0.3mm,color=red!60!blue!80,text=black]
\tikzstyle{PurplePoly}=[PurpleLine,fill=red!40!blue!50]
\tikzstyle{OrangeLine}=[thick,color=orange]
\tikzstyle{GrayLine}=[thick,color=black!50!gray]
\tikzstyle{GrayPoly}=[GrayLine,fill=gray!20]
\tikzstyle{dot}=[shape=circle,draw,color=black,fill=black,inner sep=1.5pt]
\tikzstyle{opendot}=[shape=circle,draw,color=black,fill=white,inner sep=1.5pt]
\tikzstyle{bigdot}=[dot,inner sep=2pt]
\tikzstyle{littledot}=[dot,inner sep=0.75pt]
\tikzstyle{littleopendot}=[opendot,inner sep=0.75pt]
\tikzstyle{smalldot}=[dot,inner sep=1pt]
\tikzstyle{disk}=[thick,shape=circle,draw,color=black,fill=yellow!10]
\tikzstyle{lightdisk}=[thick,shape=circle,draw,color=black!40,fill=yellow!20]
\tikzstyle{plate}=[thick,shape=rectangle,draw,color=black,fill=yellow!10,rounded corners,minimum size=1.1cm]

\tikzstyle{RowNode} = [
    shape = rectangle, 
    draw = black,
    fill = yellow!40, 
    rounded corners, 
    minimum height = 1.25cm, 
    minimum width = 1.5cm
]

\tikzstyle{BlueLine}=[line width=0.4mm,color=blue,text=black]
\tikzstyle{BluePoly}=[BlueLine,fill=blue!20]
\tikzstyle{RedLine}=[line width=0.4mm,color=red,text=black]
\tikzstyle{RedPoly}=[RedLine,fill=red!20]
    \begin{scope}
        \node[RowNode] (a1) at (0,12) {$\begin{matrix} 0 & 3 & 0 & 0 \\ 0 & 2 & 3 & 1 \\ 0 & 0 & 5 & 2 \end{matrix}$};

        \node[RowNode] (b1) at (-9,8) {$\begin{matrix} 0 & 3 & 0 & 0 \\ 0 & 2 & 8 & 3 \end{matrix}$};
        \node[RowNode] (b2) at (-3,8) {$\begin{matrix} 3 & 0 & 0 \\ 2 & 3 & 1 \\ 0 & 5 & 2 \end{matrix}$};
        \node[RowNode] (b3) at (0,8) {$\begin{matrix} 0 & 3 & 0 \\ 0 & 5 & 1 \\ 0 & 5 & 2 \end{matrix}$};
        \node[RowNode] (b4) at (3,8) {$\begin{matrix} 0 & 3 & 0 \\ 0 & 2 & 4 \\ 0 & 0 & 7 \end{matrix}$};
        \node[RowNode] (b5) at (9,8) {$\begin{matrix} 0 & 5 & 3 & 1 \\ 0 & 0 & 5 & 2 \end{matrix}$};
    
        \node[RowNode] (c1) at (-12,4) {$\begin{matrix} 3 & 0 & 0 \\ 2 & 8 & 3 \end{matrix}$};
        \node[RowNode] (c2) at (-9,4) {$\begin{matrix} 0 & 3 & 0 \\ 0 & 10 & 3 \end{matrix}$};
        \node[RowNode] (c3) at (-6,4) {$\begin{matrix} 0 & 3 & 0 \\ 0 & 2 & 11 \end{matrix}$};
        \node[RowNode] (c4) at (-3,4) {$\begin{matrix} 3 & 0 \\ 5 & 1 \\ 5 & 2 \end{matrix}$};
        \node[RowNode] (c5) at (0,4) {$\begin{matrix} 3 & 0 \\ 2 & 4 \\ 0 & 7 \end{matrix}$};
        \node[RowNode] (c6) at (3,4) {$\begin{matrix} 0 & 3 \\ 0 & 6 \\ 0 & 7 \end{matrix}$};
        \node[RowNode] (c7) at (6,4) {$\begin{matrix} 5 & 3 & 1 \\ 0 & 5 & 2 \end{matrix}$};
        \node[RowNode] (c8) at (9,4) {$\begin{matrix} 0 & 8 & 1 \\ 0 & 5 & 2 \end{matrix}$};
        \node[RowNode] (c9) at (12,4) {$\begin{matrix} 0 & 5 & 4 \\ 0 & 0 & 7 \end{matrix}$};

        \node[RowNode] (d1) at (-12,0) {$\begin{matrix} 3 & 0 \\ 10 & 3 \end{matrix}$};
        \node[RowNode] (d2) at (-9,0) {$\begin{matrix} 3 & 0 \\ 2 & 11 \end{matrix}$};
        \node[RowNode] (d3) at (-6,0) {$\begin{matrix} 0 & 3 \\ 0 & 13 \end{matrix}$};
        \node[RowNode] (d4) at (6,0) {$\begin{matrix} 8 & 1 \\ 5 & 2 \end{matrix}$};
        \node[RowNode] (d5) at (9,0) {$\begin{matrix} 5 & 4 \\ 0 & 7 \end{matrix}$};
        \node[RowNode] (d6) at (12,0) {$\begin{matrix} 0 & 9 \\ 0 & 7 \end{matrix}$};

        \draw[BlueLine] (b1) -- (a1) -- (b5);
        \draw[BlueLine] (c1) -- (b2) -- (c7);
        \draw[BlueLine] (c2) -- (b3) -- (c8);
        \draw[BlueLine] (c3) -- (b4) -- (c9);
        \draw[BlueLine] (d1) -- (c4) -- (d4);
        \draw[BlueLine] (d2) -- (c5) -- (d5);
        \draw[BlueLine] (d3) -- (c6) -- (d6);

        \draw[RedLine] (a1) -- (b2);
        \draw[RedLine] (a1) -- (b3);
        \draw[RedLine] (a1) -- (b4);
        \draw[RedLine] (b2) -- (c5) -- (b4) -- (c6) -- (b3) -- (c4) -- (b2);

        \draw[RedLine] (b1) -- (c1);
        \draw[RedLine] (b1) -- (c2);
        \draw[RedLine] (b1) -- (c3);
        \draw[RedLine] (c1) -- (d2) -- (c3) -- (d3) -- (c2) -- (d1) -- (c1);

        \draw[RedLine] (b5) -- (c7);
        \draw[RedLine] (b5) -- (c8);
        \draw[RedLine] (b5) -- (c9);
        \draw[RedLine] (c7) -- (d5) -- (c9) -- (d6) -- (c8) -- (d4) -- (c7);
    \end{scope}
\end{tikzpicture}
    \caption{The elements in $\comp_{16}(\protect\closedsquare)$
    below the fixed rectangular composition defined in
    Remark~\ref{rem:spine-biorthoscheme}. The blue edges in the 
    order diagram correspond to row-merges and the red edges correspond to column-merges.}
    \label{fig:lower-set-rect-comp}
\end{figure}

\begin{rem}[Illustrating bi-orthoschemes]
    \label{rem:spine-biorthoscheme}
    Similar to the case for multisets in a line segment presented 
    in Remark~\ref{rem:spine-orthoscheme}, it is straightforward to read 
    off the labels for the bi-orthoscheme containing a given multiset 
    $\bz \in \mult_n(\Qb)$ and all of its faces. In 
    Figure~\ref{fig:mult-rect-projections},
    we see that a multiset in $\mult_{16}(\Qb)$ and its projections 
    immediately determine the corresponding rectangular composition and 
    projected linear compositions. 
    The rectangular composition $\comp(\bz)$ labels a triangular prism in 
    the cell structure for $\mult_{16}(\Qb)$, depicted in Figure~\ref{fig:prism-rect-comp}
    as the product of a red $2$-orthoscheme with label $[0\ 5\ 8\ 3]$ (i.e. a right triangle
    with legs of length $\sqrt{5}$ and $\sqrt{8}$) and a blue $1$-orthoscheme with label
    $[3\ 6\ 7]$ (i.e. an edge with length $\sqrt{6}$). 
    For the full face poset of this cell, we can take the subposet of $\comp_{16}(\closedsquare)$ consisting of all elements
    below $\comp(\bz)$, illustrated in Figure~\ref{fig:lower-set-rect-comp}.
\end{rem}

Next, we describe the structure of the top-dimensional
cells in $\mult_n(\Qb)$ by proving Theorem~\ref{mainthm:dual-graph}. 
Recall from the Introduction that $\Gamma^{LR}(\sym_n,S)$ is the graph 
obtained by overlaying the left and right Cayley graphs of $\sym_n$ on
a common vertex set.

\begin{thm}[Theorem~\ref{mainthm:dual-graph}]\label{thm:dual-graph}
    The dual graph of the bi-simplicial cell structure on $\mult_n(\Qb)$ 
    is isomorphic to $\Gamma^{LR}(\sym_n,S)$.
\end{thm}

\begin{proof}
    The top-dimensional cells of $\mult_n(\Qb)$ are labeled by the 
    maximal elements of $\comp_n(\closedsquare)$, which are 
    the $(n+2)\times(n+2)$ matrices formed by padding an $n\times n$
    permutation matrix with zeros on all four sides. Thus, each 
    vertex of the dual graph can be uniquely labeled by an element of
    $\sym_n$, which we represent with an $n\times n$ permutation matrix.
    
    Two top-dimensional cells of $\mult_n(\Qb)$ share a face of codimension 
    one if and only if the two corresponding permutation matrices can be 
    obtained from one another via a single swap of two rows or two columns. 
    Swapping rows $i$ and $i+1$ corresponds to left multiplication by
    the permutation matrix for the transposition $(i\ i+1)$, whereas
    swapping columns $i$ and $i+1$ corresponds to right multiplication
    by the same matrix. This means that the edges of the dual graph come from
    both the left and right Cayley graphs of $\sym_n$ with respect to the
    standard generating set of adjacent transpositions, and the proof 
    is complete.
\end{proof}

To clarify any confusion regarding the two actions by $\sym_n$, consider the
following example of applying the transposition $\sigma_2 = (2\ 3)$ to the
left and right of a $5\times 5$ permutation matrix:
\[
    \sigma_2 \cdot \begin{bmatrix}
        0 & 0 & 0 & 1 & 0 \\
        1 & 0 & 0 & 0 & 0 \\
        0 & 0 & 0 & 0 & 1 \\
        0 & 1 & 0 & 0 & 0 \\
        0 & 0 & 1 & 0 & 0
    \end{bmatrix} =
    \begin{bmatrix}
        0 & 0 & 0 & 1 & 0 \\
        0 & 0 & 0 & 0 & 1 \\
        1 & 0 & 0 & 0 & 0 \\
        0 & 1 & 0 & 0 & 0 \\
        0 & 0 & 1 & 0 & 0
    \end{bmatrix}
\]
and
\[
    \begin{bmatrix}
        0 & 0 & 0 & 1 & 0 \\
        1 & 0 & 0 & 0 & 0 \\
        0 & 0 & 0 & 0 & 1 \\
        0 & 1 & 0 & 0 & 0 \\
        0 & 0 & 1 & 0 & 0
    \end{bmatrix}
    \cdot \sigma_2 = 
    \begin{bmatrix}
        0 & 0 & 0 & 1 & 0 \\
        1 & 0 & 0 & 0 & 0 \\
        0 & 0 & 0 & 0 & 1 \\
        0 & 0 & 1 & 0 & 0 \\
        0 & 1 & 0 & 0 & 0
    \end{bmatrix}.
\]
Each matrix can also be viewed with a representative 
generic configuration of five dots in $\Qb$; in this
perspective, keeping in mind the matrix orientation 
described in Remark~\ref{rem:rotate}, the left
action permutes the ordering of points in the $x$ direction,
whereas the right action permutes the $y$ coordinates.
In Figure~\ref{fig:dual-graph}, we provide an illustration
of the full dual graph for $\mult_3(\Qb)$.

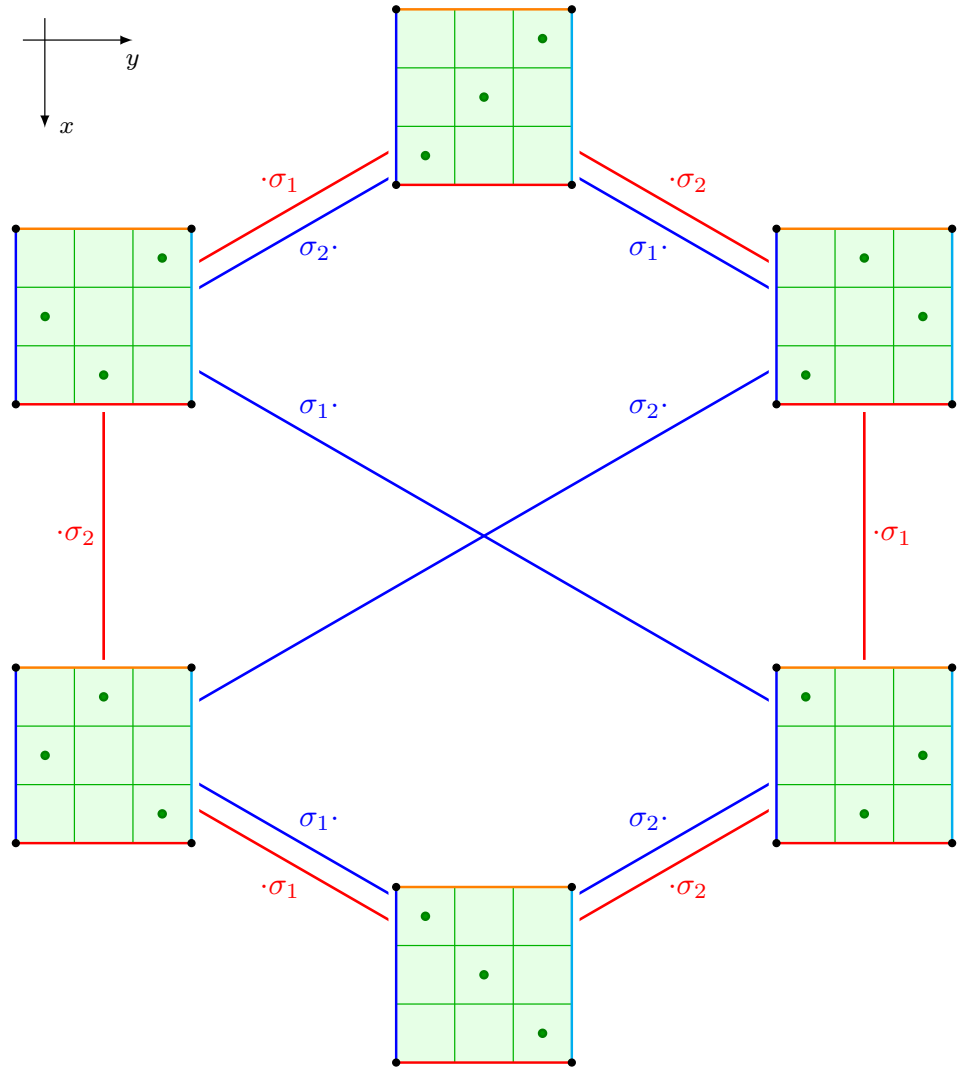
\begin{figure}
    \centering
    \begin{tikzpicture}

\tikzstyle{BlueLine}=[line width=0.3mm,color=blue,text=black]
\tikzstyle{BluePoly}=[BlueLine,fill=blue!20]
\tikzstyle{RedLine}=[line width=0.3mm,color=red,text=black]
\tikzstyle{RedPoly}=[RedLine,fill=red!20]
\tikzstyle{GreenLine}=[thick,color=black!30!green,text=black]
\tikzstyle{GreenPoly}=[thick,color=green!50!black,fill=green!60!black,join=bevel]
\tikzstyle{dot}=[shape=circle,draw,color=black,fill=black,inner sep=1.5pt]
\tikzstyle{bigdot}=[dot,inner sep=.8pt]
\tikzstyle{littledot}=[dot,inner sep=.8pt]

  \newcommand{\trigrid}{
    \newcommand{\s}{1}
    \draw[color=white,line width=5pt]  (-\s,-\s) -- (\s,-\s) -- (\s,\s) -- (-\s,\s) -- cycle;
    \filldraw[color=green!10!white] (-\s,-\s) -- (\s,-\s) -- (\s,\s) -- (-\s,\s) -- cycle;
    \foreach \x in {-\s/3,\s/3} {
      \draw[GreenLine,thin] (\x,-\s) -- (\x,\s);
      \draw[GreenLine,thin] (-\s,\x) -- (\s,\x);
    }
    \draw[BlueLine,thick,color=blue] (-\s,\s) -- (-\s,-\s);
    \draw[BlueLine,thick,color=cyan] (\s,\s) -- (\s,-\s);
    \draw[RedLine,thick,color=orange] (-\s,\s) -- (\s,\s);
    \draw[RedLine,thick,color=red] (-\s,-\s) -- (\s,-\s);
    \foreach \x in {-\s,\s} {
      \foreach \y in {-\s,\s} {
	\node[littledot,color=black] () at (\x,\y) {};	
      }
    }
  }
  
  \newcommand{\drawdots}[3]{
    \node[bigdot,GreenPoly] () at (2*\s/3*#1-4*\s/3,2*\s/3) {};
    \node[bigdot,GreenPoly] () at (2*\s/3*#2-4*\s/3,0) {};
    \node[bigdot,GreenPoly] () at (2*\s/3*#3-4*\s/3,-2*\s/3) {};
  }
  
  \def\E{5} 
  \def\e{4.7}
  \coordinate (bot) at (270:\E); \coordinate (bote) at (270:\e);
  \coordinate (a1) at (210:\E); \coordinate (a1e) at (210:\e);
  \coordinate (a2) at (330:\E); \coordinate (a2e) at (330:\e);
  \coordinate (b1) at (150:\E); \coordinate (b1e) at (150:\e);
  \coordinate (b2) at (30:\E); \coordinate (b2e) at (30:\e);
  \coordinate (top) at (90:\E); \coordinate (tope) at (90:\e);

  \def\a{\textcolor{red}{$\cdot\sigma_1$}} 
  \def\b{\textcolor{red}{$\cdot\sigma_2$}} 
  \def\c{\textcolor{blue}{$\sigma_1\cdot$}} 
  \def\d{\textcolor{blue}{$\sigma_2\cdot$}} 
  \def\f{9pt} 
  \draw[RedLine] (bot)--(a1) node[midway,shift={(240:\f)}] {\a};
  \draw[RedLine] (bot)--(a2) node[midway,shift={(300:\f)}] {\b};
  \draw[RedLine] (a1)--(b1) node[midway,shift={(180:\f)}] {\b};
  \draw[RedLine] (a2)--(b2) node[midway,shift={(0:\f)}] {\a};
  \draw[RedLine] (b1)--(top) node[midway,shift={(120:\f)}] {\a};
  \draw[RedLine] (b2)--(top) node[midway,shift={(60:\f)}] {\b};
  \draw[BlueLine] (bote)--(a1e) node[midway,shift={(60:\f)}] {\c};
  \draw[BlueLine] (bote)--(a2e) node[midway,shift={(120:\f)}] {\d};
  \draw[BlueLine] (a1e)--(b2e) node[pos=0.75,shift={(120:\f)}] {\d};
  \draw[BlueLine] (a2e)--(b1e) node[pos=0.75,shift={(60:\f)}] {\c};
  \draw[BlueLine] (b1e)--(tope) node[midway,shift={(300:\f)}] {\d};
  \draw[BlueLine] (b2e)--(tope) node[midway,shift={(240:\f)}] {\c};

  \begin{scope}[shift={(bot)}]
    \trigrid
    \drawdots{1}{2}{3}
  \end{scope}
  
  \begin{scope}[shift={(a1)}]
    \trigrid
    \drawdots{2}{1}{3}
  \end{scope}
  
  \begin{scope}[shift={(a2)}]
    \trigrid	
    \drawdots{1}{3}{2}
  \end{scope}
  
  \begin{scope}[shift={(b1)}]
    \trigrid	
    \drawdots{3}{1}{2}
  \end{scope}
  
  \begin{scope}[shift={(b2)}]
    \trigrid	
    \drawdots{2}{3}{1}
  \end{scope}
  
  \begin{scope}[shift={(top)}]
    \trigrid	
    \drawdots{3}{2}{1}
  \end{scope}

    \begin{scope}[shift={(-5,5.65)}]
        \draw[-latex] (-0.25,0) -- (1,0);
        \draw[-latex] (0,0.25) -- (0,-1);
        \node () at (1,-0.25) {\footnotesize $y$};
        \node () at (0.25,-1) {\footnotesize $x$};
    \end{scope}

\end{tikzpicture}
    \caption{The dual graph for $\mult_3(\Qb)$ can be viewed
    as the superposition of the left and right Cayley graphs for 
    $\sym_3$ with respect to the standard generating set 
    $\{\sigma_1,\sigma_2\}$. The vertices of the graph are labeled 
    by generic configurations representing the top-dimensional cells
    in $\mult_3(\Qb)$, from which the corresponding permutation matrix 
    can be read off by replacing dots with ones and blank squares 
    with zeros. The axes in the upper left serve as a reminder that 
    each vertex is illustrated using a plane which
    has been rotated to accommodate the matrix orientation.}
    \label{fig:dual-graph}
\end{figure}

A more robust picture can be obtained by looking at the full dual complex
of $\mult_n(\Qb)$ rather than just the dual graph. If $\Ib^n$ is an
$n$-dimensional cube, subdivided into $n!$ standard $n$-dimensional 
orthoschemes as in Definition~\ref{def:cubes}, then the dual complex
of this cell structure is a simple convex polytope called the
\emph{$(n-1)$-dimensional permutahedron}, and it can be defined
as the convex hull of the orbit of a generic configuration under
the (left) action of $\sym_n$ on $\Ib^n$. The $1$-skeleton of the 
permutahedron is then the (right) Cayley graph of $\sym_n$ with
respect to the standard generating set of adjacent transpositions.

For example, if $\Ib = [0,4]$ and $\bx = (1,2,3)$, then $\Ib^3$ 
is subdivided into six standard 3-dimensional orthoschemes, each of
which contains one point in the orbit of $\bx$ under the (left) action 
of $\sym_3$ in its interior. All six points lie on the plane
$x + y + z = 6$, and their convex hull is a regular hexagon with sides
of length $\sqrt{2}$. The vertices and edges of this hexagon form
the (right) Cayley graph of $\sym_3$ with respect to the generating set
$\{\sigma_1,\sigma_2\}$.

It follows that the product $\Qb^n = \Ib^n \times \Jb^n$ has a dual complex
which is the direct product of two $(n-1)$-dimensional permutahedra:
a blue permutahedron dual to the simplicial structure on $\Ib^n$ and
a red permutahedron dual to the simplicial structure on $\Jb^n$. 
The coordinate-permuting action of $\sym_n$ on $\Qb^n$ corresponds to 
the diagonal action on the product $\Ib^n \times \Jb^n$, with 
$\mult_n(\Qb)$ as the quotient. The dual complex for $\mult_n(\Qb)$ is 
thus the quotient of the product of two $(n-1)$-dimensional permutahedra
by the diagonal action of $\sym_n$.

\section{Bi-orthoscheme spines}\label{sec:spines}

In Definition~\ref{def:orthoscheme}, we defined the \emph{spine} of an
$n$-dimensional orthoscheme to be the subgraph of its $1$-skeleton
which contains all $n+1$ vertices, along with the $n$ edges which
connect one vertex to the next in the linear ordering.  In
Definition~\ref{def:bisimplices}, we defined the \emph{spine} of a
bi-orthoscheme to be the subcomplex of its $2$-skeleton which is
formed by the product of the spines of the two factor
simplices. Geometrically, the spine of an orthoscheme is a line
segment, and the spine of a bi-orthoscheme is a rectangle.  See
Figure~\ref{fig:spine-biorthoscheme} for an example, noting how the
spine can be overlaid on top of a representative rectangular multiset
$\bz$ so that each $2\times 2$ matrix $\begin{bmatrix} a & b \\ c &
  d \end{bmatrix}$ (labeling a vertex of the bi-orthoscheme spine) is
placed at a point where there are $a$ elements of $\bz$ to the upper
left, $b$ to the upper right, $c$ to the lower left, and $d$ to the
lower right. In this manner, the entire spine of a bi-orthoscheme in
$\mult_n(\Qb)$ can be easily read off of any point in its interior.

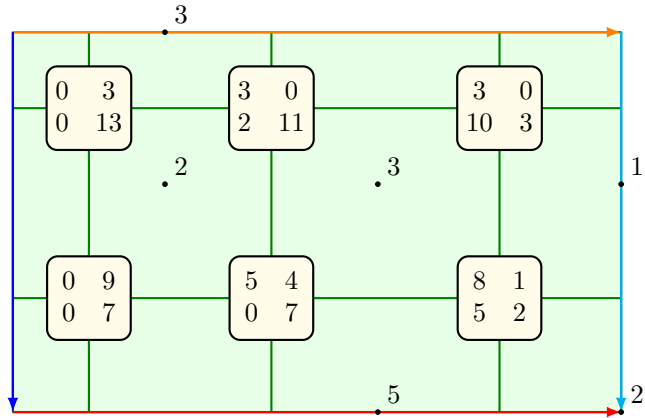
\begin{figure}
    \centering
    \begin{tikzpicture}

\tikzstyle{BlueLine}=[line width=0.3mm,color=blue,text=black]
\tikzstyle{BluePoly}=[BlueLine,fill=blue!20]
\tikzstyle{RedLine}=[line width=0.3mm,color=red,text=black]
\tikzstyle{RedPoly}=[RedLine,fill=red!20]
\tikzstyle{GreenLine}=[thick,color=green!50!black,text=black]
\tikzstyle{GreenPoly}=[thick,color=green!50!black,fill=green!30,join=bevel]
\tikzstyle{PurpleLine}=[line width=0.3mm,color=red!60!blue!80,text=black]
\tikzstyle{PurplePoly}=[PurpleLine,fill=red!40!blue!50]
\tikzstyle{OrangeLine}=[thick,color=orange]
\tikzstyle{GrayLine}=[thick,color=black!50!gray]
\tikzstyle{GrayPoly}=[GrayLine,fill=gray!20]
\tikzstyle{dot}=[shape=circle,draw,color=black,fill=black,inner sep=1.5pt]
\tikzstyle{opendot}=[shape=circle,draw,color=black,fill=white,inner sep=1.5pt]
\tikzstyle{bigdot}=[dot,inner sep=2pt]
\tikzstyle{littledot}=[dot,inner sep=0.75pt]
\tikzstyle{littleopendot}=[opendot,inner sep=0.75pt]
\tikzstyle{smalldot}=[dot,inner sep=1pt]
\tikzstyle{disk}=[thick,shape=circle,draw,color=black,fill=yellow!10]
\tikzstyle{lightdisk}=[thick,shape=circle,draw,color=black!40,fill=yellow!20]
\tikzstyle{plate}=[thick,shape=rectangle,draw,color=black,fill=yellow!10,rounded corners,minimum size=1.1cm]


\def\shade{.8}

\definecolor{c00}{rgb}{0,\shade,\shade} 
\definecolor{c60}{rgb}{0,0,\shade}
\definecolor{c120}{rgb}{\shade,0,\shade} 
\definecolor{c180}{rgb}{\shade,0,0}
\definecolor{c240}{rgb}{\shade,\shade,0} 
\definecolor{c300}{rgb}{0,\shade,0}
\definecolor{c360}{rgb}{0,\shade,\shade}

\colorlet{c10}{c60!16.7!c00} \colorlet{c20}{c60!33.3!c00}
\colorlet{c30}{c60!50!c00}   \colorlet{c40}{c60!66.7!c00}
\colorlet{c50}{c60!83.3!c00}

\colorlet{c70}{c120!16.7!c60} \colorlet{c80}{c120!33.3!c60}
\colorlet{c90}{c120!50!c60}   \colorlet{c100}{c120!66.7!c60}
\colorlet{c110}{c120!83.3!c60}

\colorlet{c130}{c180!16.7!c120} \colorlet{c140}{c180!33.3!c120}
\colorlet{c150}{c180!50!c120}   \colorlet{c160}{c180!66.7!c120}
\colorlet{c170}{c180!83.3!c120}

\colorlet{c190}{c240!16.7!c180} \colorlet{c200}{c240!33.3!c180}
\colorlet{c210}{c240!50!c180}   \colorlet{c220}{c240!66.7!c180}
\colorlet{c230}{c240!83.3!c180}

\colorlet{c250}{c300!16.7!c240} \colorlet{c260}{c300!33.3!c240}
\colorlet{c270}{c300!50!c240}   \colorlet{c280}{c300!66.7!c240}
\colorlet{c290}{c300!83.3!c240}

\colorlet{c310}{c360!16.7!c300} \colorlet{c320}{c360!33.3!c300}
\colorlet{c330}{c360!50!c300}   \colorlet{c340}{c360!66.7!c300}
\colorlet{c350}{c360!83.3!c300}

\colorlet{ca}{c00}
\colorlet{cb}{c60}
\colorlet{cc}{c180}
\colorlet{cd}{c300}

\colorlet{bckgrd}{green!10!white}
\colorlet{cA}{black!10!white}
\colorlet{cB}{black!20!white}
\colorlet{cC}{black!40!white}
\colorlet{cD}{black!60!white}
\colorlet{cE}{black!80!white}

  \def\xl{0}
  \def\xr{8}
  \def\xma{2}
  \def\xmb{4.8}
  \def\yb{0}
  \def\yt{5}
  \def\yma{3}
  \coordinate (xl) at (\xl,0);
  \coordinate (xr) at (\xr,0);
  \coordinate (xma) at (\xma,0);
  \coordinate (xmb) at (\xmb,0);
  \coordinate (yb) at (0,\yb);
  \coordinate (yt) at (0,\yt);
  \coordinate (yma) at (0,\yma);
  \coordinate (v1) at (\xr,\yt);
  \coordinate (v2) at (\xl,\yt);
  \coordinate (v3) at (\xl,\yb);
  \coordinate (v4) at (\xr,\yb);

    \coordinate (z1) at (\xma,\yma);
    \coordinate (z2) at (\xma,\yt);
    \coordinate (z3) at (\xmb,\yb);
    \coordinate (z4) at (\xmb,\yma);
    \coordinate (z5) at (\xr,\yb);
    \coordinate (z6) at (\xr,\yma);

    \filldraw[fill=white, color=white] ($(v1)+(1,1)$) -- ($(v2)+(-1,1)$) -- ($(v3)+(-1,-1)$) -- ($(v4)+(1,-1)$) -- cycle;
    
    \filldraw[GreenPoly, fill = bckgrd] (v1)--(v2)--(v3)--(v4)--cycle;

    \draw[GreenLine] ($(\xl,\yb)!0.5!(\xma,\yb)$) -- ($(\xl,\yt)!0.5!(\xma,\yt)$);
    \draw[GreenLine] ($(\xma,\yb)!0.5!(\xmb,\yb)$) -- ($(\xma,\yt)!0.5!(\xmb,\yt)$);
    \draw[GreenLine] ($(\xmb,\yb)!0.5!(\xr,\yb)$) -- ($(\xmb,\yt)!0.5!(\xr,\yt)$);

    \draw[GreenLine] ($(\xl,\yb)!0.5!(\xl,\yma)$) -- ($(\xr,\yb)!0.5!(\xr,\yma)$);
    \draw[GreenLine] ($(\xl,\yma)!0.5!(\xl,\yt)$) -- ($(\xr,\yma)!0.5!(\xr,\yt)$);

    \draw[RedLine,-latex,color=orange] (v2)--(v1);
    \draw[RedLine,-latex,color=red] (v3)--(v4);
    \draw[BlueLine,-latex,color=cyan] (v1)--(v4);
    \draw[BlueLine,-latex,color=blue] (v2)--(v3);

    \node[plate] () at ($(\xl,\yb)!0.5!(\xma,\yma)$) {
        $\begin{matrix}
            0 & 9 \\ 0 & 7 
        \end{matrix}$
    };
    \node[plate] () at ($(\xma,\yb)!0.5!(\xmb,\yma)$) {
        $\begin{matrix}
            5 & 4 \\ 0 & 7 
        \end{matrix}$
    };
    \node[plate] () at ($(\xmb,\yb)!0.5!(\xr,\yma)$) {
        $\begin{matrix}
            8 & 1 \\ 5 & 2
        \end{matrix}$
    };
    \node[plate] () at ($(\xl,\yma)!0.5!(\xma,\yt)$) {
        $\begin{matrix}
            0 & 3 \\ 0 & 13
        \end{matrix}$
    };
    \node[plate] () at ($(\xma,\yma)!0.5!(\xmb,\yt)$) {
        $\begin{matrix}
            3 & 0 \\ 2 & 11
        \end{matrix}$
    };
    \node[plate] () at ($(\xmb,\yma)!0.5!(\xr,\yt)$) {
        $\begin{matrix}
            3 & 0 \\ 10 & 3
        \end{matrix}$
    };
    
    \foreach \p in {z1,z2,z3,z4,z5,z6}{
      \filldraw[dot,color=black] (\p) circle (.3mm);}
    
    \node[anchor=south west] at (z1) {$2$};
    \node[anchor=south west] at (z2) {$3$};
    \node[anchor=south west] at (z3) {$5$};
    \node[anchor=south west] at (z4) {$3$};
    \node[anchor=south west] at (z5) {$2$};
    \node[anchor=south west] at (z6) {$1$};

\end{tikzpicture}
    \caption{The multiset $\bz \in \mult_{16}(\Qb)$ defined in 
    Remark~\ref{rem:spine-biorthoscheme}, with the corresponding
    bi-orthoscheme spine (which consists of six vertices, seven edges,
    and two rectangles) overlaid. The same spine is indicated with
    bold edges in Figure~\ref{fig:prism-rect-comp}.}
    \label{fig:spine-biorthoscheme}
\end{figure}

The spines of top-dimensional bi-orthoschemes are worth discussing separately.

\begin{rem}[Top-dimensional spines]\label{rem:top-spines}
See Figure~\ref{fig:generic} for a generic configuration $\bz$ of $4$ points in 
$\Qb$ and observe that it lies in the interior of the bi-orthoscheme in
$\mult_4(\Qb)$ which is labeled by the rectangular composition
\[
    \comp(\bz) = 
    \begin{bmatrix}
        0 & 0 & 0 & 0 & 0 & 0 \\
        0 & 0 & 1 & 0 & 0 & 0 \\
        0 & 0 & 0 & 1 & 0 & 0 \\
        0 & 1 & 0 & 0 & 0 & 0 \\
        0 & 0 & 0 & 0 & 1 & 0 \\
        0 & 0 & 0 & 0 & 0 & 0
    \end{bmatrix}
\]
in $\comp_4(\closedsquare)$. In the same figure, we see the spine of
the bi-orthoscheme labeled by 
$\comp(\bz)$ with the four elements of $\bz$ superimposed.  Note in
particular that each edge of the spine corresponds to moving $1$ from
one entry to another in the same row or column, so each edge has a
well-defined color depending on which column or row was left fixed:
blue for fixing the left column, cyan for the right column, red for
the bottom row, and orange for the top row. Furthermore, we can
observe that the points in $\bz$ correspond to the squares in the
spine with four distinct edge colors.
\end{rem}

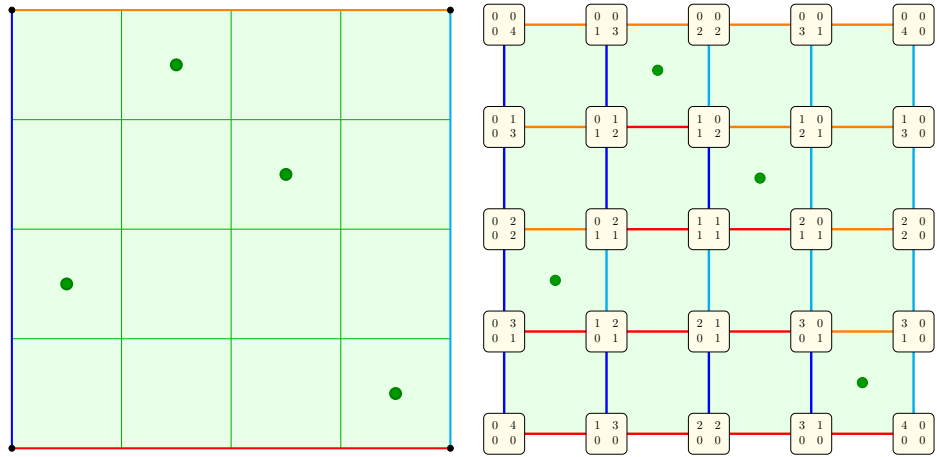
\begin{figure}
    \centering
    \begin{tikzpicture}

\tikzstyle{BlueLine}=[line width=0.3mm,color=blue,text=black]
\tikzstyle{BluePoly}=[BlueLine,fill=blue!20]
\tikzstyle{RedLine}=[line width=0.3mm,color=red,text=black]
\tikzstyle{RedPoly}=[RedLine,fill=red!20]
\tikzstyle{GreenLine}=[thick,color=black!30!green,text=black]
\tikzstyle{GreenPoly}=[thick,color=green!50!black,fill=green!60!black,join=bevel]
\tikzstyle{OrangeLine}=[thick,color=orange]
\tikzstyle{GrayLine}=[thick,color=black!50!gray]
\tikzstyle{GrayPoly}=[GrayLine,fill=gray!20]
\tikzstyle{dot}=[shape=circle,draw,color=black,fill=black,inner sep=1.5pt]
\tikzstyle{bigdot}=[dot,inner sep=1.5pt]
\tikzstyle{littledot}=[dot,inner sep=.8pt]
\tikzstyle{disk}=[thick,shape=circle,draw,color=black,fill=yellow!10]
\tikzstyle{plate}=[thick,shape=rectangle,draw,color=black,fill=yellow!10,
	rounded corners,minimum size=1.1cm]

\tikzstyle{dot}=[shape=circle,draw,color=black,fill=black,inner sep=1.5pt]
\tikzstyle{opendot}=[dot,fill=white]
\tikzstyle{disk}=[thick,shape=circle,draw,color=black]

  \newcommand{\trigrid}{
    \newcommand{\s}{3cm}
    \draw[color=white,line width=5pt]  (-\s,-\s) -- (\s,-\s) -- (\s,\s) -- (-\s,\s) -- cycle;
    \filldraw[color=green!10!white] (-\s,-\s) -- (\s,-\s) -- (\s,\s) -- (-\s,\s) -- cycle;
    \foreach \x in {-\s/2,0,\s/2} {
      \draw[GreenLine,thin] (\x,-\s) -- (\x,\s);
      \draw[GreenLine,thin] (-\s,\x) -- (\s,\x);
    }
    \draw[BlueLine,thick,color=blue] (-\s,\s) -- (-\s,-\s);
    \draw[BlueLine,thick,color=cyan] (\s,\s) -- (\s,-\s);
    \draw[RedLine,thick,color=orange] (-\s,\s) -- (\s,\s);
    \draw[RedLine,thick,color=red] (-\s,-\s) -- (\s,-\s);
    \foreach \x in {-\s,\s} {
      \foreach \y in {-\s,\s} {
	\node[littledot,color=black] () at (\x,\y) {};	
      }
    }
  }
  
  \newcommand{\drawdots}[4]{
    \node[bigdot,GreenPoly] (a1) at (2*\s/4*#1-5*\s/4,3*\s/4) {};
    \node[bigdot,GreenPoly] (a2) at (2*\s/4*#2-5*\s/4,1*\s/4) {};
    \node[bigdot,GreenPoly] (a3) at (2*\s/4*#3-5*\s/4,-1*\s/4) {};
    \node[bigdot,GreenPoly] (a4) at (2*\s/4*#4-5*\s/4,-3*\s/4) {};
    \draw (a1) node[anchor=south] {};
    \draw (a2) node[anchor=south] {};
    \draw (a3) node[anchor=south] {};
    \draw (a4) node[anchor=south] {};
  }

  \trigrid
  \drawdots{2}{3}{1}{4}

\end{tikzpicture} \ \
    \begin{tikzpicture}

    \tikzstyle{BlueLine}=[line width=0.3mm,color=blue,text=black]
    \tikzstyle{BluePoly}=[BlueLine,fill=blue!20]
    \tikzstyle{RedLine}=[line width=0.3mm,color=red,text=black]
    \tikzstyle{RedPoly}=[RedLine,fill=red!20]
    \tikzstyle{GreenLine}=[thick,color=black!30!green,text=black]
    \tikzstyle{GreenPoly}=[thick,color=green!50!black,fill=green!60!black,join=bevel]
    \tikzstyle{OrangeLine}=[thick,color=orange]
    \tikzstyle{GrayLine}=[thick,color=black!50!gray]
    \tikzstyle{GrayPoly}=[GrayLine,fill=gray!20]
    \tikzstyle{dot}=[shape=circle,draw,color=black,fill=black,inner sep=1.5pt]
    \tikzstyle{bigdot}=[dot,inner sep=3pt]
    \tikzstyle{littledot}=[dot,inner sep=.8pt]
    \tikzstyle{disk}=[thick,shape=circle,draw,color=black,fill=yellow!10]
    \tikzstyle{plate}=[thick,shape=rectangle,draw,color=black,fill=yellow!10,
    	rounded corners,minimum size=1.2cm]
    
    \tikzstyle{dot}=[shape=circle,draw,color=black,fill=black,inner sep=1.5pt]
    \tikzstyle{opendot}=[dot,fill=white]
    \tikzstyle{disk}=[thick,shape=circle,draw,color=black]

  \def\s{3 cm}
  \def\S{2*\s}
  \def\r{.7cm}
  \def\R{0.5*\r}

  \node (b11) at (-2*\s,2*\s) {}; 
  \node (b12) at (-1*\s,2*\s) {};
  \node (b13) at (0*\s,2*\s) {}; 
  \node (b14) at (1*\s,2*\s) {}; 
  \node (b15) at (2*\s,2*\s) {}; 
  \node (b21) at (-2*\s,\s) {}; 
  \node (b22) at (-1*\s,\s) {}; 
  \node (b23) at (0*\s,\s) {}; 
  \node (b24) at (1*\s,\s) {}; 
  \node (b25) at (2*\s,\s) {}; 
  \node (b31) at (-2*\s,0*\s) {}; 
  \node (b32) at (-1*\s,0*\s) {}; 
  \node (b33) at (0*\s,0*\s) {}; 
  \node (b34) at (1*\s,0*\s) {}; 
  \node (b35) at (2*\s,0*\s) {}; 
  \node (b41) at (-2*\s,-1*\s) {}; 
  \node (b42) at (-1*\s,-1*\s) {}; 
  \node (b43) at (0*\s,-1*\s) {}; 
  \node (b44) at (1*\s,-1*\s) {}; 
  \node (b45) at (2*\s,-1*\s) {}; 
  \node (b51) at (-2*\s,-2*\s) {}; 
  \node (b52) at (-1*\s,-2*\s) {}; 
  \node (b53) at (0*\s,-2*\s) {}; 
  \node (b54) at (1*\s,-2*\s) {}; 
  \node (b55) at (2*\s,-2*\s) {};

    \filldraw[color=green!10!white] (-2*\s,-2*\s) -- (2*\s,-2*\s) -- 
    (2*\s,2*\s) -- (-2*\s,2*\s) -- cycle;

    \begin{scope}[line width = 2pt, color=red]

        \draw (b22) -- (b23);
        \draw (b32) -- (b33) -- (b34);

        \draw (b41) -- (b42) -- (b43) -- (b44);

        \draw (b51) -- (b52) -- (b53) -- (b54) -- (b55);
    \end{scope}

    \begin{scope}[line width = 2pt, color=orange]
        \draw (b11) -- (b12) -- (b13) -- (b14) -- (b15);

        \draw (b21) -- (b22);
        \draw (b23) -- (b24) -- (b25);

        \draw (b31) -- (b32);
        \draw (b34) -- (b35);

        \draw (b44) -- (b45);
    \end{scope}

    \begin{scope}[line width = 2pt, color=blue]
        \draw (b11) -- (b21) -- (b31) -- (b41) -- (b51);

        \draw (b12) -- (b22) -- (b32);
        \draw (b42) -- (b52);

        \draw (b23) -- (b33);
        \draw (b43) -- (b53);

        \draw (b44) -- (b54);
    \end{scope}

    \begin{scope}[line width = 2pt, color=cyan]
        \draw (b32) -- (b42);

        \draw (b13) -- (b23);
        \draw (b33) -- (b43);

        \draw (b14) -- (b24) -- (b34) -- (b44);

        \draw (b15) -- (b25) -- (b35) -- (b45) -- (b55);
    \end{scope}


\node[plate] () at (b11) {$\begin{matrix} 0 & 0 \\ 0 & 4 \end{matrix}$};
\node[plate] () at (b12) {$\begin{matrix} 0 & 0 \\ 1 & 3 \end{matrix}$};
\node[plate] () at (b13) {$\begin{matrix} 0 & 0 \\ 2 & 2 \end{matrix}$};
\node[plate] () at (b14) {$\begin{matrix} 0 & 0 \\ 3 & 1 \end{matrix}$};
\node[plate] () at (b15) {$\begin{matrix} 0 & 0 \\ 4 & 0 \end{matrix}$};

\node[plate] () at (b21) {$\begin{matrix} 0 & 1 \\ 0 & 3 \end{matrix}$};
\node[plate] () at (b31) {$\begin{matrix} 0 & 2 \\ 0 & 2 \end{matrix}$};
\node[plate] () at (b41) {$\begin{matrix} 0 & 3 \\ 0 & 1 \end{matrix}$};
\node[plate] () at (b51) {$\begin{matrix} 0 & 4 \\ 0 & 0 \end{matrix}$};

\node[plate] () at (b52) {$\begin{matrix} 1 & 3 \\ 0 & 0 \end{matrix}$};
\node[plate] () at (b53) {$\begin{matrix} 2 & 2 \\ 0 & 0 \end{matrix}$};
\node[plate] () at (b54) {$\begin{matrix} 3 & 1 \\ 0 & 0 \end{matrix}$};
\node[plate] () at (b55) {$\begin{matrix} 4 & 0 \\ 0 & 0 \end{matrix}$};

\node[plate] () at (b25) {$\begin{matrix} 1 & 0 \\ 3 & 0 \end{matrix}$};
\node[plate] () at (b35) {$\begin{matrix} 2 & 0 \\ 2 & 0 \end{matrix}$};
\node[plate] () at (b45) {$\begin{matrix} 3 & 0 \\ 1 & 0 \end{matrix}$};

\node[plate] () at (b22) {$\begin{matrix} 0 & 1 \\ 1 & 2 \end{matrix}$};

\node[plate] () at (b23) {$\begin{matrix} 1 & 0 \\ 1 & 2 \end{matrix}$};
\node[plate] () at (b24) {$\begin{matrix} 1 & 0 \\ 2 & 1 \end{matrix}$};

\node[plate] () at (b32) {$\begin{matrix} 0 & 2 \\ 1 & 1 \end{matrix}$};

\node[plate] () at (b42) {$\begin{matrix} 1 & 2 \\ 0 & 1 \end{matrix}$};

\node[plate] () at (b43) {$\begin{matrix} 2 & 1 \\ 0 & 1 \end{matrix}$};

\node[plate] () at (b44) {$\begin{matrix} 3 & 0 \\ 0 & 1 \end{matrix}$};

\node[plate] () at (b34) {$\begin{matrix} 2 & 0 \\ 1 & 1 \end{matrix}$};

\node[plate] () at (b33) {$\begin{matrix} 1 & 1 \\ 1 & 1 \end{matrix}$};

\node[bigdot,GreenPoly, anchor = south] (a1) at (-0.5*\s,1.5*\s) {};
\node[bigdot,GreenPoly] (a2) at (-1.5*\s,-0.5*\s) {};
\node[bigdot,GreenPoly] (a3) at (0.5*\s,0.5*\s) {};
\node[bigdot,GreenPoly] (a4) at (1.5*\s,-1.5*\s) {};
\draw (a1) node[anchor=south] {};
\draw (a2) node[anchor=south] {};
\draw (a3) node[anchor=south] {};
\draw (a4) node[anchor=south] {};

\end{tikzpicture}
    \caption{On the left, a generic configuration $\bz$ of 
    four distinct points in $\Qb$. On the right, the same configuration
    along with the spine of its corresponding bi-orthoscheme overlaid.}
    \label{fig:generic}
\end{figure}

Finally, we note that all of the spines for the top-dimensional
bi-orthoschemes are visible in a single $3$-dimensional graph.

\begin{rem}[Spines in a tetrahedral graph]
  In Section~\ref{sec:pt-rect} we noted that the vertices of
  $\mult_n(\Qb)$ correspond to the positive integer lattice points in
  the hyperplane $x_1 + x_2 + x_3 + x_4 = n$.  In
  Figure~\ref{fig:multiset-spine} we have placed the vertices
  according to their location in this tetrahedron with $n=4$, and we
  have drawn some edges connecting them.  In particular, we have drawn an edge
  when two vertex labels differ by moving one unit from an entry in the 
  $2 \times 2$ matrix to another in the same row or column.
  This means that the $1$-skeleton of the spine of any top-dimensional bi-orthoscheme 
  is visible in this graph (Remark~\ref{rem:top-spines}).  For example, the spine from 
  Figure~\ref{fig:generic} is highlighted in Figure~\ref{fig:multiset-spine-highlight}.
  Notice that the vertices labeled by $2\times 2$ matrices with either a row 
  or a column of zeros form a cyclic graph with $4n$ edges, and that all 
  of the top-dimensional spines contain this cycle of edges.
\end{rem}

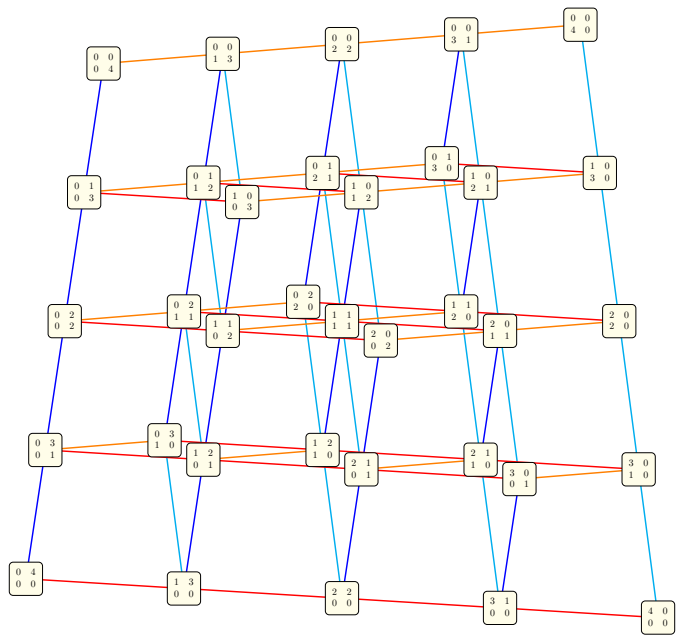
\begin{figure}
    \centering
    \begin{tikzpicture}

    \tikzstyle{BlueLine}=[line width=0.3mm,color=blue,text=black]
    \tikzstyle{BluePoly}=[BlueLine,fill=blue!20]
    \tikzstyle{RedLine}=[line width=0.3mm,color=red,text=black]
    \tikzstyle{RedPoly}=[RedLine,fill=red!20]
    \tikzstyle{GreenLine}=[thick,color=black!30!green,text=black]
    \tikzstyle{GreenPoly}=[thick,color=green!50!black,fill=green!60!black,join=bevel]
    \tikzstyle{OrangeLine}=[thick,color=orange]
    \tikzstyle{GrayLine}=[thick,color=black!50!gray]
    \tikzstyle{GrayPoly}=[GrayLine,fill=gray!20]
    \tikzstyle{dot}=[shape=circle,draw,color=black,fill=black,inner sep=1.5pt]
    \tikzstyle{bigdot}=[dot,inner sep=.8pt]
    \tikzstyle{littledot}=[dot,inner sep=.8pt]
    \tikzstyle{disk}=[thick,shape=circle,draw,color=black,fill=yellow!10]
    \tikzstyle{plate}=[thick,shape=rectangle,draw,color=black,fill=yellow!10,
    	rounded corners,minimum size=1.2cm]
    
    \tikzstyle{dot}=[shape=circle,draw,color=black,fill=black,inner sep=1.5pt]
    \tikzstyle{opendot}=[dot,fill=white]
    \tikzstyle{disk}=[thick,shape=circle,draw,color=black]

  \def\s{5 cm}
  \def\S{2*\s}
  \def\r{.7cm}
  \def\R{0.5*\r}

  \node (a11) at (-2*\s,2*\s) {}; 
  \node (a12) at (-1*\s,2*\s) {};
  \node (a13) at (0*\s,2*\s) {}; 
  \node (a14) at (1*\s,2*\s) {}; 
  \node (a15) at (2*\s,2*\s) {}; 
  \node (a21) at (-2*\s,\s) {}; 
  \node (a22) at (-1*\s,\s) {}; 
  \node (a23) at (0*\s,\s) {}; 
  \node (a24) at (1*\s,\s) {}; 
  \node (a25) at (2*\s,\s) {}; 
  \node (a31) at (-2*\s,0*\s) {}; 
  \node (a32) at (-1*\s,0*\s) {}; 
  \node (a33) at (0*\s,0*\s) {}; 
  \node (a34) at (1*\s,0*\s) {}; 
  \node (a35) at (2*\s,0*\s) {}; 
  \node (a41) at (-2*\s,-1*\s) {}; 
  \node (a42) at (-1*\s,-1*\s) {}; 
  \node (a43) at (0*\s,-1*\s) {}; 
  \node (a44) at (1*\s,-1*\s) {}; 
  \node (a45) at (2*\s,-1*\s) {}; 
  \node (a51) at (-2*\s,-2*\s) {}; 
  \node (a52) at (-1*\s,-2*\s) {}; 
  \node (a53) at (0*\s,-2*\s) {}; 
  \node (a54) at (1*\s,-2*\s) {}; 
  \node (a55) at (2*\s,-2*\s) {};
  
  \path (a11) +(2*\r,-2*\R) coordinate (b11);
  \path (a12) +(1*\r,-1*\R) coordinate (b12);
  \path (a13) +(0*\r,0*\R) coordinate (b13);
  \path (a14) +(-1*\r,1*\R) coordinate (b14);
  \path (a15) +(-2*\r,2*\R) coordinate (b15);
  
  \path (a21) +(1*\r,-1*\R) coordinate (b21);
  \path (a22) +(0*\r,0*\R) coordinate (b22a);
  \path (a22) +(2*\r,-2*\R) coordinate (b22b);
  \path (a23) +(-1*\r,1*\R) coordinate (b23a);
  \path (a23) +(1*\r,-1*\R) coordinate (b23b);
  \path (a24) +(-2*\r,2*\R) coordinate (b24a);
  \path (a24) +(0*\r,0*\R) coordinate (b24b);
  \path (a25) +(-1*\r,1*\R) coordinate (b25);
  
  \path (a31) +(0*\r,0*\R) coordinate (b31) {};
  \path (a32) +(-1*\r,1*\R) coordinate (b32a) {};
  \path (a32) +(1*\r,-1*\R) coordinate (b32b) {};
  \path (a33) +(-2*\r,2*\R) coordinate (b33a) {};
  \path (a33) +(0*\r,0*\R) coordinate (b33b) {};
  \path (a33) +(2*\r,-2*\R) coordinate (b33c) {};
  \path (a34) +(-1*\r,1*\R) coordinate (b34a) {};
  \path (a34) +(1*\r,-1*\R) coordinate (b34b) {};
  \path (a35) +(0*\r,0*\R) coordinate (b35) {};

  \path (a41) +(-1*\r,1*\R) coordinate (b41) {};
  \path (a42) +(-2*\r,2*\R) coordinate (b42a) {};
  \path (a42) +(0*\r,0*\R) coordinate (b42b) {};
  \path (a43) +(-1*\r,1*\R) coordinate (b43a) {};
  \path (a43) +(1*\r,-1*\R) coordinate (b43b) {};
  \path (a44) +(0*\r,0*\R) coordinate (b44a) {};
  \path (a44) +(2*\r,-2*\R) coordinate (b44b) {};
  \path (a45) +(1*\r,-1*\R) coordinate (b45) {};

  \path (a51) +(-2*\r,2*\R) coordinate (b51) {};
  \path (a52) +(-1*\r,1*\R) coordinate (b52) {};
  \path (a53) +(0*\r,0*\R) coordinate (b53) {};
  \path (a54) +(1*\r,-1*\R) coordinate (b54) {};
  \path (a55) +(2*\r,-2*\R) coordinate (b55) {};


    \begin{scope}[line width = 1.5pt, color=blue]
        \draw (b11) -- (b21) -- (b31) -- (b41) -- (b51);

        \draw (b12) -- (b22a) -- (b32a) -- (b42a);
        \draw (b22b) -- (b32b) -- (b42b) -- (b52);

        \draw (b13) -- (b23a) -- (b33a);
        \draw (b23b) -- (b33b) -- (b43a);
        \draw (b33c) -- (b43b) -- (b53);

        \draw (b14) -- (b24a);
        \draw (b24b) -- (b34a);
        \draw (b34b) -- (b44a);
        \draw (b44b) -- (b54);
    \end{scope}

    \begin{scope}[line width = 1.5pt, color=cyan]
        \draw (b12) -- (b22b);
        \draw (b22a) -- (b32b);
        \draw (b32a) -- (b42b);
        \draw (b42a) -- (b52);

        \draw (b13) -- (b23b) -- (b33c);
        \draw (b23a) -- (b33b) -- (b43b);
        \draw (b33a) -- (b43a) -- (b53);

        \draw (b14) -- (b24b) -- (b34b) -- (b44b);
        \draw (b24a) -- (b34a) -- (b44a) -- (b54);

        \draw (b15) -- (b25) -- (b35) -- (b45) -- (b55);
    \end{scope}

    \begin{scope}[line width = 1.5pt, color=red]
        \draw (b21) -- (b22b);
        \draw (b22a) -- (b23b);
        \draw (b23a) -- (b24b);
        \draw (b24a) -- (b25);

        \draw (b31) -- (b32b) -- (b33c);
        \draw (b32a) -- (b33b) -- (b34b);
        \draw (b33a) -- (b34a) -- (b35);

        \draw (b41) -- (b42b) -- (b43b) -- (b44b);
        \draw (b42a) -- (b43a) -- (b44a) -- (b45);

        \draw (b51) -- (b52) -- (b53) -- (b54) -- (b55);
    \end{scope}

    \begin{scope}[line width = 1.5pt, color=orange]
        \draw (b11) -- (b12) -- (b13) -- (b14) -- (b15);

        \draw (b21) -- (b22a) -- (b23a) -- (b24a);
        \draw (b22b) -- (b23b) -- (b24b) -- (b25);

        \draw (b31) -- (b32a) -- (b33a);
        \draw (b32b) -- (b33b) -- (b34a);
        \draw (b33c) -- (b34b) -- (b35);

        \draw (b41) -- (b42a);
        \draw (b42b) -- (b43a);
        \draw (b43b) -- (b44a);
        \draw (b44b) -- (b45);
    \end{scope}


\node[plate] () at (b11) {$\begin{matrix} 0 & 0 \\ 0 & 4 \end{matrix}$};
\node[plate] () at (b12) {$\begin{matrix} 0 & 0 \\ 1 & 3 \end{matrix}$};
\node[plate] () at (b13) {$\begin{matrix} 0 & 0 \\ 2 & 2 \end{matrix}$};
\node[plate] () at (b14) {$\begin{matrix} 0 & 0 \\ 3 & 1 \end{matrix}$};
\node[plate] () at (b15) {$\begin{matrix} 0 & 0 \\ 4 & 0 \end{matrix}$};

\node[plate] () at (b21) {$\begin{matrix} 0 & 1 \\ 0 & 3 \end{matrix}$};
\node[plate] () at (b31) {$\begin{matrix} 0 & 2 \\ 0 & 2 \end{matrix}$};
\node[plate] () at (b41) {$\begin{matrix} 0 & 3 \\ 0 & 1 \end{matrix}$};
\node[plate] () at (b51) {$\begin{matrix} 0 & 4 \\ 0 & 0 \end{matrix}$};

\node[plate] () at (b52) {$\begin{matrix} 1 & 3 \\ 0 & 0 \end{matrix}$};
\node[plate] () at (b53) {$\begin{matrix} 2 & 2 \\ 0 & 0 \end{matrix}$};
\node[plate] () at (b54) {$\begin{matrix} 3 & 1 \\ 0 & 0 \end{matrix}$};
\node[plate] () at (b55) {$\begin{matrix} 4 & 0 \\ 0 & 0 \end{matrix}$};

\node[plate] () at (b25) {$\begin{matrix} 1 & 0 \\ 3 & 0 \end{matrix}$};
\node[plate] () at (b35) {$\begin{matrix} 2 & 0 \\ 2 & 0 \end{matrix}$};
\node[plate] () at (b45) {$\begin{matrix} 3 & 0 \\ 1 & 0 \end{matrix}$};

\node[plate] () at (b22a) {$\begin{matrix} 0 & 1 \\ 1 & 2 \end{matrix}$};
\node[plate] () at (b22b) {$\begin{matrix} 1 & 0 \\ 0 & 3 \end{matrix}$};

\node[plate] () at (b23a) {$\begin{matrix} 0 & 1 \\ 2 & 1 \end{matrix}$};
\node[plate] () at (b23b) {$\begin{matrix} 1 & 0 \\ 1 & 2 \end{matrix}$};

\node[plate] () at (b24a) {$\begin{matrix} 0 & 1 \\ 3 & 0 \end{matrix}$};
\node[plate] () at (b24b) {$\begin{matrix} 1 & 0 \\ 2 & 1 \end{matrix}$};

\node[plate] () at (b32a) {$\begin{matrix} 0 & 2 \\ 1 & 1 \end{matrix}$};
\node[plate] () at (b32b) {$\begin{matrix} 1 & 1 \\ 0 & 2 \end{matrix}$};

\node[plate] () at (b42a) {$\begin{matrix} 0 & 3 \\ 1 & 0 \end{matrix}$};
\node[plate] () at (b42b) {$\begin{matrix} 1 & 2 \\ 0 & 1 \end{matrix}$};

\node[plate] () at (b43a) {$\begin{matrix} 1 & 2 \\ 1 & 0 \end{matrix}$};
\node[plate] () at (b43b) {$\begin{matrix} 2 & 1 \\ 0 & 1 \end{matrix}$};

\node[plate] () at (b44a) {$\begin{matrix} 2 & 1 \\ 1 & 0 \end{matrix}$};
\node[plate] () at (b44b) {$\begin{matrix} 3 & 0 \\ 0 & 1 \end{matrix}$};

\node[plate] () at (b34a) {$\begin{matrix} 1 & 1 \\ 2 & 0 \end{matrix}$};
\node[plate] () at (b34b) {$\begin{matrix} 2 & 0 \\ 1 & 1 \end{matrix}$};

\node[plate] () at (b33a) {$\begin{matrix} 0 & 2 \\ 2 & 0 \end{matrix}$};
\node[plate] () at (b33b) {$\begin{matrix} 1 & 1 \\ 1 & 1 \end{matrix}$};
\node[plate] () at (b33c) {$\begin{matrix} 2 & 0 \\ 0 & 2 \end{matrix}$};



\end{tikzpicture}
    \caption{The tetrahedral graph of the multiset space $\mult_4(\Qb)$.}
    \label{fig:multiset-spine}
\end{figure}

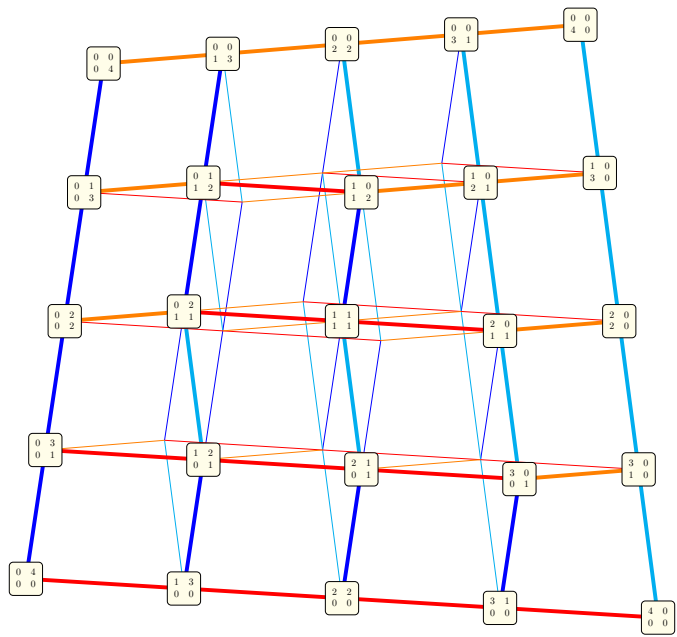
\begin{figure}
    \centering
    \begin{tikzpicture}

    \tikzstyle{BlueLine}=[line width=0.3mm,color=blue,text=black]
    \tikzstyle{BluePoly}=[BlueLine,fill=blue!20]
    \tikzstyle{RedLine}=[line width=0.3mm,color=red,text=black]
    \tikzstyle{RedPoly}=[RedLine,fill=red!20]
    \tikzstyle{GreenLine}=[thick,color=black!30!green,text=black]
    \tikzstyle{GreenPoly}=[thick,color=green!50!black,fill=green!60!black,join=bevel]
    \tikzstyle{OrangeLine}=[thick,color=orange]
    \tikzstyle{GrayLine}=[thick,color=black!50!gray]
    \tikzstyle{GrayPoly}=[GrayLine,fill=gray!20]
    \tikzstyle{dot}=[shape=circle,draw,color=black,fill=black,inner sep=1.5pt]
    \tikzstyle{bigdot}=[dot,inner sep=.8pt]
    \tikzstyle{littledot}=[dot,inner sep=.8pt]
    \tikzstyle{disk}=[thick,shape=circle,draw,color=black,fill=yellow!10]
    \tikzstyle{plate}=[thick,shape=rectangle,draw,color=black,fill=yellow!10,
    	rounded corners,minimum size=1.2cm]
    
    \tikzstyle{dot}=[shape=circle,draw,color=black,fill=black,inner sep=1.5pt]
    \tikzstyle{opendot}=[dot,fill=white]
    \tikzstyle{disk}=[thick,shape=circle,draw,color=black]

  \def\s{5 cm}
  \def\S{2*\s}
  \def\r{.7cm}
  \def\R{0.5*\r}

  \node (a11) at (-2*\s,2*\s) {}; 
  \node (a12) at (-1*\s,2*\s) {};
  \node (a13) at (0*\s,2*\s) {}; 
  \node (a14) at (1*\s,2*\s) {}; 
  \node (a15) at (2*\s,2*\s) {}; 
  \node (a21) at (-2*\s,\s) {}; 
  \node (a22) at (-1*\s,\s) {}; 
  \node (a23) at (0*\s,\s) {}; 
  \node (a24) at (1*\s,\s) {}; 
  \node (a25) at (2*\s,\s) {}; 
  \node (a31) at (-2*\s,0*\s) {}; 
  \node (a32) at (-1*\s,0*\s) {}; 
  \node (a33) at (0*\s,0*\s) {}; 
  \node (a34) at (1*\s,0*\s) {}; 
  \node (a35) at (2*\s,0*\s) {}; 
  \node (a41) at (-2*\s,-1*\s) {}; 
  \node (a42) at (-1*\s,-1*\s) {}; 
  \node (a43) at (0*\s,-1*\s) {}; 
  \node (a44) at (1*\s,-1*\s) {}; 
  \node (a45) at (2*\s,-1*\s) {}; 
  \node (a51) at (-2*\s,-2*\s) {}; 
  \node (a52) at (-1*\s,-2*\s) {}; 
  \node (a53) at (0*\s,-2*\s) {}; 
  \node (a54) at (1*\s,-2*\s) {}; 
  \node (a55) at (2*\s,-2*\s) {};
  
  \path (a11) +(2*\r,-2*\R) coordinate (b11);
  \path (a12) +(1*\r,-1*\R) coordinate (b12);
  \path (a13) +(0*\r,0*\R) coordinate (b13);
  \path (a14) +(-1*\r,1*\R) coordinate (b14);
  \path (a15) +(-2*\r,2*\R) coordinate (b15);
  
  \path (a21) +(1*\r,-1*\R) coordinate (b21);
  \path (a22) +(0*\r,0*\R) coordinate (b22a);
  \path (a22) +(2*\r,-2*\R) coordinate (b22b);
  \path (a23) +(-1*\r,1*\R) coordinate (b23a);
  \path (a23) +(1*\r,-1*\R) coordinate (b23b);
  \path (a24) +(-2*\r,2*\R) coordinate (b24a);
  \path (a24) +(0*\r,0*\R) coordinate (b24b);
  \path (a25) +(-1*\r,1*\R) coordinate (b25);
  
  \path (a31) +(0*\r,0*\R) coordinate (b31) {};
  \path (a32) +(-1*\r,1*\R) coordinate (b32a) {};
  \path (a32) +(1*\r,-1*\R) coordinate (b32b) {};
  \path (a33) +(-2*\r,2*\R) coordinate (b33a) {};
  \path (a33) +(0*\r,0*\R) coordinate (b33b) {};
  \path (a33) +(2*\r,-2*\R) coordinate (b33c) {};
  \path (a34) +(-1*\r,1*\R) coordinate (b34a) {};
  \path (a34) +(1*\r,-1*\R) coordinate (b34b) {};
  \path (a35) +(0*\r,0*\R) coordinate (b35) {};

  \path (a41) +(-1*\r,1*\R) coordinate (b41) {};
  \path (a42) +(-2*\r,2*\R) coordinate (b42a) {};
  \path (a42) +(0*\r,0*\R) coordinate (b42b) {};
  \path (a43) +(-1*\r,1*\R) coordinate (b43a) {};
  \path (a43) +(1*\r,-1*\R) coordinate (b43b) {};
  \path (a44) +(0*\r,0*\R) coordinate (b44a) {};
  \path (a44) +(2*\r,-2*\R) coordinate (b44b) {};
  \path (a45) +(1*\r,-1*\R) coordinate (b45) {};

  \path (a51) +(-2*\r,2*\R) coordinate (b51) {};
  \path (a52) +(-1*\r,1*\R) coordinate (b52) {};
  \path (a53) +(0*\r,0*\R) coordinate (b53) {};
  \path (a54) +(1*\r,-1*\R) coordinate (b54) {};
  \path (a55) +(2*\r,-2*\R) coordinate (b55) {};


    \begin{scope}[line width = 1pt, color=blue]
        \draw[line width = 4pt] (b11) -- (b21) -- (b31) -- (b41) -- (b51);

        \draw[line width = 4pt] (b12) -- (b22a) -- (b32a);
        \draw (b32a) -- (b42a);
        \draw (b22b) -- (b32b) -- (b42b);
        \draw[line width = 4pt] (b42b) -- (b52);

        \draw (b13) -- (b23a) -- (b33a);
        \draw[line width = 4pt] (b23b) -- (b33b);
        \draw (b33b) -- (b43a);
        \draw (b33c) -- (b43b);
        \draw[line width = 4pt] (b43b) -- (b53);

        \draw (b14) -- (b24a);
        \draw (b24b) -- (b34a);
        \draw (b34b) -- (b44a);
        \draw[line width = 4pt] (b44b) -- (b54);
    \end{scope}
    
    \begin{scope}[line width = 1pt, color=cyan]
        \draw (b12) -- (b22b);
        \draw (b22a) -- (b32b);
        \draw[line width = 4pt] (b32a) -- (b42b);
        \draw (b42a) -- (b52);

        \draw[line width = 4pt] (b13) -- (b23b);
        \draw (b23b) -- (b33c);
        \draw (b23a) -- (b33b);
        \draw[line width = 4pt] (b33b) -- (b43b);
        \draw (b33a) -- (b43a) -- (b53);

        \draw[line width = 4pt] (b14) -- (b24b) -- (b34b) -- (b44b);
        \draw (b24a) -- (b34a) -- (b44a) -- (b54);

        \draw[line width = 4pt] (b15) -- (b25) -- (b35) -- (b45) -- (b55);
    \end{scope}
    
    \begin{scope}[line width = 1pt, color=red]
        \draw (b21) -- (b22b);
        \draw[line width = 4pt] (b22a) -- (b23b);
        \draw (b23a) -- (b24b);
        \draw (b24a) -- (b25);

        \draw (b31) -- (b32b) -- (b33c);
        \draw[line width = 4pt] (b32a) -- (b33b) -- (b34b);
        \draw (b33a) -- (b34a) -- (b35);

        \draw[line width = 4pt] (b41) -- (b42b) -- (b43b) -- (b44b);
        \draw (b42a) -- (b43a) -- (b44a) -- (b45);

        \draw[line width = 4pt] (b51) -- (b52) -- (b53) -- (b54) -- (b55);
    \end{scope}

    \begin{scope}[line width = 1pt, color=orange]
        \draw[line width = 4pt] (b11) -- (b12) -- (b13) -- (b14) -- (b15);

        \draw[line width = 4pt] (b21) -- (b22a);
        \draw (b22a) -- (b23a) -- (b24a);
        \draw (b22b) -- (b23b);
        \draw[line width = 4pt] (b23b) -- (b24b) -- (b25);

        \draw[line width = 4pt] (b31) -- (b32a);
        \draw (b32a) -- (b33a);
        \draw (b32b) -- (b33b) -- (b34a);
        \draw (b33c) -- (b34b);
        \draw[line width = 4pt] (b34b) -- (b35);

        \draw (b41) -- (b42a);
        \draw (b42b) -- (b43a);
        \draw (b43b) -- (b44a);
        \draw[line width = 4pt] (b44b) -- (b45);
    \end{scope}


\node[plate] () at (b11) {$\begin{matrix} 0 & 0 \\ 0 & 4 \end{matrix}$};
\node[plate] () at (b12) {$\begin{matrix} 0 & 0 \\ 1 & 3 \end{matrix}$};
\node[plate] () at (b13) {$\begin{matrix} 0 & 0 \\ 2 & 2 \end{matrix}$};
\node[plate] () at (b14) {$\begin{matrix} 0 & 0 \\ 3 & 1 \end{matrix}$};
\node[plate] () at (b15) {$\begin{matrix} 0 & 0 \\ 4 & 0 \end{matrix}$};

\node[plate] () at (b21) {$\begin{matrix} 0 & 1 \\ 0 & 3 \end{matrix}$};
\node[plate] () at (b31) {$\begin{matrix} 0 & 2 \\ 0 & 2 \end{matrix}$};
\node[plate] () at (b41) {$\begin{matrix} 0 & 3 \\ 0 & 1 \end{matrix}$};
\node[plate] () at (b51) {$\begin{matrix} 0 & 4 \\ 0 & 0 \end{matrix}$};

\node[plate] () at (b52) {$\begin{matrix} 1 & 3 \\ 0 & 0 \end{matrix}$};
\node[plate] () at (b53) {$\begin{matrix} 2 & 2 \\ 0 & 0 \end{matrix}$};
\node[plate] () at (b54) {$\begin{matrix} 3 & 1 \\ 0 & 0 \end{matrix}$};
\node[plate] () at (b55) {$\begin{matrix} 4 & 0 \\ 0 & 0 \end{matrix}$};

\node[plate] () at (b25) {$\begin{matrix} 1 & 0 \\ 3 & 0 \end{matrix}$};
\node[plate] () at (b35) {$\begin{matrix} 2 & 0 \\ 2 & 0 \end{matrix}$};
\node[plate] () at (b45) {$\begin{matrix} 3 & 0 \\ 1 & 0 \end{matrix}$};

\node[plate] () at (b22a) {$\begin{matrix} 0 & 1 \\ 1 & 2 \end{matrix}$};

\node[plate] () at (b23b) {$\begin{matrix} 1 & 0 \\ 1 & 2 \end{matrix}$};

\node[plate] () at (b24b) {$\begin{matrix} 1 & 0 \\ 2 & 1 \end{matrix}$};

\node[plate] () at (b32a) {$\begin{matrix} 0 & 2 \\ 1 & 1 \end{matrix}$};

\node[plate] () at (b42b) {$\begin{matrix} 1 & 2 \\ 0 & 1 \end{matrix}$};

\node[plate] () at (b43b) {$\begin{matrix} 2 & 1 \\ 0 & 1 \end{matrix}$};

\node[plate] () at (b44b) {$\begin{matrix} 3 & 0 \\ 0 & 1 \end{matrix}$};

\node[plate] () at (b34b) {$\begin{matrix} 2 & 0 \\ 1 & 1 \end{matrix}$};

\node[plate] () at (b33b) {$\begin{matrix} 1 & 1 \\ 1 & 1 \end{matrix}$};



\end{tikzpicture}
    \caption{The $1$-skeleton of the top-dimensional spine from Figure~\ref{fig:generic}
    inside the tetrahedral graph of $\mult_4(\Qb)$.}
    \label{fig:multiset-spine-highlight}
\end{figure}

\bibliographystyle{amsalpha}
\bibliography{multisets}

@Unpublished{dm-cnc,
  author = {Dougherty, Michael and McCammond, Jon},
  title = {Continuous noncrossing partitions and weighted circular factorizations},
  note = {Preprint 2025. \texttt{arXiv:2507.00283}},
}

@unpublished{gcp2,
    author = {Dougherty, Michael and McCammond, Jon},
    title = {Geometric combinatorics of polynomials {II}: {P}olynomials and cell structures},
    note = {Preprint 2024. \texttt{arXiv:2410.03047}},
}

@book {LaZv04,
    AUTHOR = {Lando, Sergei K. and Zvonkin, Alexander K.},
     TITLE = {Graphs on surfaces and their applications},
    SERIES = {Encyclopaedia of Mathematical Sciences},
    VOLUME = {141},
      NOTE = {With an appendix by Don B. Zagier,
              Low-Dimensional Topology, II},
 PUBLISHER = {Springer-Verlag, Berlin},
      YEAR = {2004},
     PAGES = {xvi+455},
      ISBN = {3-540-00203-0},
   MRCLASS = {14H55 (05-02 05C10 05C30 05C50 14H10 14H30 32G15)},
  MRNUMBER = {2036721},
MRREVIEWER = {Athanase Papadopoulos},
       DOI = {10.1007/978-3-540-38361-1},
       URL = {https://doi.org/10.1007/978-3-540-38361-1},
}

@incollection {wachs07,
    AUTHOR = {Wachs, Michelle L.},
     TITLE = {Poset topology: tools and applications},
 BOOKTITLE = {Geometric combinatorics},
    SERIES = {IAS/Park City Math. Ser.},
    VOLUME = {13},
     PAGES = {497--615},
 PUBLISHER = {Amer. Math. Soc., Providence, RI},
      YEAR = {2007},
      ISBN = {978-0-8218-3736-8; 0-8218-3736-2},
   MRCLASS = {06B30 (05E10 52C35 55R80)},
  MRNUMBER = {2383132},
       DOI = {10.1090/pcms/013/09},
       URL = {https://doi-org.ezproxy.lafayette.edu/10.1090/pcms/013/09},
}

@article {tuffley02,
    AUTHOR = {Tuffley, Christopher},
     TITLE = {Finite subset spaces of {$S^1$}},
   JOURNAL = {Algebr. Geom. Topol.},
  FJOURNAL = {Algebraic \& Geometric Topology},
    VOLUME = {2},
      YEAR = {2002},
     PAGES = {1119--1145},
      ISSN = {1472-2747,1472-2739},
   MRCLASS = {54B20 (55Q52 55R80 57M25)},
  MRNUMBER = {1998017},
MRREVIEWER = {Jacob\ Mostovoy},
       DOI = {10.2140/agt.2002.2.1119},
       URL = {https://doi.org/10.2140/agt.2002.2.1119},
}

\end{document}